\documentclass[]{interact}

\usepackage{epstopdf}% To incorporate .eps illustrations using PDFLaTeX, etc.
\usepackage[caption=false]{subfig}% Support for small, `sub' figures and tables

\usepackage{float}
\usepackage[numbers,sort&compress]{natbib}% Citation support using natbib.sty
\bibpunct[, ]{[}{]}{,}{n}{,}{,}% Citation support using natbib.sty
\makeatletter% @ becomes a letter
\def\NAT@def@citea{\def\@citea{\NAT@separator}}% Suppress spaces between citations using natbib.sty
\makeatother% @ becomes a symbol again

\theoremstyle{plain}% Theorem-like structures provided by amsthm.sty
\newtheorem{theorem}{Theorem}[section]
\newtheorem{lemma}[theorem]{Lemma}
\newtheorem{corollary}[theorem]{Corollary}

\theoremstyle{definition}
\newtheorem{definition}[theorem]{Definition}
\newtheorem{example}[theorem]{Example}

\theoremstyle{remark}
\newtheorem{remark}{Remark}

\begin{document}

\articletype{Original Research Article}% Specify the article type or omit as appropriate

\title{Univalence of horizontal shear of Ces\`aro type transforms}

\author{
\name{Swadesh Kumar Sahoo\textsuperscript{a}\thanks{CONTACT S.~K. Sahoo. Email: swadesh.sahoo@iiti.ac.in} and Sheetal Wankhede\textsuperscript{b}}
\affil{\textsuperscript{a,b}Department of Mathematics, Indian Institute of Technology Indore, Simrol, Indore 453552, India}
}

\maketitle

\begin{abstract}
This manuscript investigates the classical problem of determining conditions on the parameters $\alpha,\beta \in \mathbb{C}$ for which the integral transform 
	$$C_{\alpha\beta}[\varphi](z):=\int_{0}^{z} \bigg(\frac{\varphi(\zeta)}{\zeta (1-\zeta)^{\beta}}\bigg)^\alpha\,d\zeta
	$$
is also univalent in the unit disk, where $\varphi$ is a normalized univalent function.
Additionally, whenever $\varphi$ belongs to some subclasses of the class of univalent functions, the univalence features of the harmonic mappings corresponding to $C_{\alpha\beta}[\varphi]$ and its rotations are derived. 
As applications to our primary findings, a few non-trivial univalent harmonic mappings are also provided.  The primary tools employed in this manuscript are Becker's univalence criteria and the shear construction developed by Clunie and Sheil-Small.
\end{abstract}


\section{Introduction}\label{IntroductionSection}
Let $\mathcal{A}$ denote the class of all analytic functions $\varphi$ in the open unit disk 
$\mathbb{D}=\lbrace z\in \mathbb{C}:\, |z|<1\rbrace$ with the normalization $\varphi(0)=0$ 
and $\varphi'(0)=1$. The subclass $\mathcal{S}$ of 
$\mathcal{A}$ consists of all univalent functions in $\mathbb{D}$. 
A function $\varphi\in\mathcal{A}$ is said to be {\em starlike of order $\delta$}, $0\le \delta<1$, if it satisfies 
${\rm Re}\,[z\varphi'(z)/\varphi(z)]>\delta$ for all $z\in\mathbb{D}$, and
is said to be {\em convex} if ${\rm Re}\,[1+z\varphi''(z)/\varphi'(z)]>0$
for all $z\in\mathbb{D}$.
The subclass of $\mathcal{S}$ made up of starlike functions of order $\delta$ 
is denoted by the symbol $\mathcal{S}^*(\delta)$.
It should be noted that a function $\varphi$ is referred to as {\em starlike} 
if it is a member of $\mathcal{S}^*(0)=:\mathcal{S}^*$.
We designate the class of convex univalent functions by 
$\mathcal{K}$.  
A function $\varphi\in\mathcal{A}$ is known as {\em close-to-convex} if and only if
$\int_{\theta_1}^{\theta_2}{\rm Re}\,[1+z\varphi''(z)/\varphi'(z)]\,d\theta>-\pi,~z=re^{i\theta}$,
for each $r\in(0,1)$ and for each pair of real numbers $\theta_1,\theta_2$ with $\theta_1<\theta_2$.
The class of close-to-convex functions is denoted by $\mathcal{CC}$.
It is well-known that $\mathcal{K}\subsetneq \mathcal{S}^*\subsetneq \mathcal{CC}\subsetneq \mathcal{S}$.

The traditional Alexander Theorem, which asserts that $\varphi\in \mathcal{S^{*}}$ if and only if $J[\varphi]\in \mathcal{K}$, where the Alexander transform $J[\varphi]$ of $\varphi\in \mathcal{A}$ defined as 
$$
J[\varphi](z)=\int_{0}^{z}\frac{\varphi(\zeta)}{\zeta}\, d\zeta,
$$
provides an important relationship between the classes $\mathcal{S^{*}}$ and $\mathcal{K}$. 
According to \cite[\S 8.4]{Dur83}, if $\varphi\in\mathcal{S}$, then $J[\varphi]$ is not always in $\mathcal{S}$. This provides impetus to research 
the preserving properties of the Alexander and related transforms of classical classes of univalent functions; see for instance \cite{KS20} and references therein. The Alexander transform was initially generalized to the following form (see \cite{Cau67,Cau71,MW71,Nun69}) in order to investigate the univalence characteristics of the integral transforms of the aforementioned kind:
$$
J_{\alpha}[\varphi](z)=\int_{0}^{z}\Big(\frac{\varphi(\zeta)}{\zeta}\Big)^{\alpha}\, d\zeta,
\quad \alpha\in\mathbb{C}.
$$
Note that $J_{1}[\varphi]=J[\varphi]$ and $J_{\alpha}[\varphi]=(I_{\alpha} \circ J)[\varphi]$, 
where $I_{\alpha}[\varphi]$ is
the Hornich scalar multiplication operator of a {\em locally univalent function} $\varphi$ (i.e. $\varphi'(z)\neq 0$) in $\mathbb{D}$ defined by
$$
I_{\alpha}[\varphi](z)=(\alpha \star \varphi(z))=\int_{0}^{z} \lbrace \varphi'(\zeta) \rbrace^{\alpha}\, d\zeta.
$$
The operator $J_{\alpha}[\varphi]$ was later considered by Kim and Merkes \cite{KM72}, and they showed that $J_{\alpha}(\mathcal{S})\subset \mathcal{S}$ for $|\alpha|\leq 1/4$. Further, the complete range of 
$\alpha$ for $J_{\alpha}(\mathcal{S})\subset \mathcal{S}$
was found by Aksent'ev and Nezhmetdinov \cite{AN87}. For the univalence of the 
operator $I_\alpha[\varphi]$, the ranges of $\alpha$ are obtained in \cite{Pfa75,Roy65}
whenever $\varphi$ is an analytic univalent function. Moreover, 
for the meromorphic univalent functions $\varphi$, conditions on $\alpha$ are obtained in \cite{NP03} for which $I_\alpha[\varphi]$ is also meromorphic univalent.
Readers can also see the work of Ponnusamy and Singh \cite{PS96} for
the univalence properties of the transforms $I_\alpha[\varphi]$ and $J_\alpha[\varphi]$
when $\varphi$ varies over other classical subclasses of $\mathcal{S}$.
It is worth noting that the univalence 
of the transforms $I_\alpha[\varphi]$ and $J_\alpha[\varphi]$ generate numerous examples of 
integral transforms which are indeed univalent.

In addition to the significance of the Alexander transform in the context of univalency, the Ces\`aro transform of $\varphi\in\mathcal{A}$, which is defined by 
$$
C[\varphi](z)=\int_{0}^{z} \frac{\varphi(\zeta)}{\zeta (1-\zeta)}\, d\zeta,
$$
has also been taken into account (see \cite{HM74}).
It is worth recalling that if $\varphi\in\mathcal{S}$ then $C[\varphi]$ may not be in
$\mathcal{S}$, see \cite[Theorem~3]{HM74}. 
Furthermore, in view of \cite[p.~424]{HM74}, the Koebe function illustrates that the starlike functions need not be preserved by the Ces\`aro transform. 
However, it is proved that the transform $C[\varphi]$ preserves the class $\mathcal{K}$; see \cite[Theorem~1]{HM74}. 
This fact encourages us to investigate the univalence properties of a generalised integral transform that incorporates both the Alexander and Ces\`aro transforms,
which is defined by
\begin{equation}\label{Eq1.1P1}
C_{\alpha\beta}[\varphi](z)=J_\alpha[\varphi]\oplus I_{\alpha\beta}[\chi]=\int_{0}^{z} \Big(\frac{\varphi(\zeta)}{\zeta (1-\zeta)^{\beta}}\Big)^{\alpha}\, d\zeta, \quad \alpha,\beta\in \mathbb{C},
\end{equation}
where $\chi(z)=-\log(1-z)$ with a suitable branch. Here, $\oplus$ denotes the Hornich
addition operator defined by
$$
(\varphi\oplus \psi)(z)=\int_0^z \varphi'(\zeta)\,\psi'(\zeta)\,d\zeta
$$
between $\varphi,\psi\in \mathcal{A}$ with $\varphi'(z)\neq 0$ and $\psi'(z)\neq 0$.
It is important to note that the operator $C_{\alpha\beta}[\varphi]$ is equivalent to the form having the integrand $(\varphi(\zeta)/\zeta)^\alpha(1-\zeta)^{-\delta}$ for some $\delta\in\mathbb{C}$. In our case, $\delta=\alpha\beta$.
We write $C_\alpha[\varphi]:=C_{\alpha 1}[\varphi]$.
Consequently, it should be noticed that $C_{11}[\varphi]=C_1[\varphi]=C[\varphi]$, $C_{\alpha 0}[\varphi]=J_\alpha[\varphi]$, and 
$C_{\alpha\beta}[\varphi]=(I_{\alpha}\circ C_{\beta})[\varphi]$, where 
$$
C_{\beta}[\varphi](z)=\int_{0}^{z} \frac{\varphi(\zeta)}{\zeta (1-\zeta)^{\beta}}\, d\zeta, \quad \beta\in\mathbb{C}.
$$
While $\varphi$ varies over specific subclasses of $\mathcal{S}$, the analytic and geometric properties of $C_\beta[\varphi]$ have been explored in \cite{KS20-1,KS20,PSS19}.

The major objective of this manuscript is to deepen our understanding of the univalence of Ces\`aro type 
integral transforms of analytic functions to the harmonic setting. 
Let $\mathbb{H}$ denote the class of all harmonic mappings $f=h+\overline{g}$
in $\mathbb{D}$ with the normalization $h(0)=g(0)=0$ and $h'(0)=1$.
Here, the functions $h$ and $g$ are called the {\em analytic} and the {\em co-analytic} parts of $f$, respectively. 
The notations
$$
\mathcal{S}_{\mathbb{H}}=\{f\in \mathbb{H}:\,\mbox{ $f$ is univalent in $\mathbb{D}$}\}
~\mbox{ and }~
\mathcal{CC}_{\mathbb{H}}=\{f\in \mathbb{H}:\,\mbox{ $f$ is close-to-convex in $\mathbb{D}$}\},
$$
respectively, represent the class of harmonic univalent and
harmonic close-to-convex mappings in $\mathbb{D}$. 
Here, $f\in \mathbb{H}$ is called a close-to-convex function if $f(\mathbb{D})$
is a close-to-convex domain \cite{CS84}.
Note that
$\mathcal{CC}_{\mathbb{H}}\subsetneq \mathcal{S}_{\mathbb{H}}$. 
Now we recall that a complex-valued harmonic mapping 
$f=h+\overline{g}$ defined on a simply connected domain $\Omega$
is called {\em locally univalent} if the Jacobian of $f$ defined by $J_{f}=|h'|^2-|g'|^2$ is non-vanishing.
Further, it is called sense-preserving if $J_{f}>0$, or equivalently,
the second complex dilatation $\omega=g'/h'$ has the property that $|\omega(z)|<1$ in $\Omega$, see \cite{Lew36}.
In this context, $f=h+\overline{g}$ is called the {\em horizontal shear} of $h-g=:\varphi$ with its dilatation $\omega=g'/h'$.
For this purpose, one can use the method of shear construction as a tool to construct univalent harmonic mappings that are convex in same direction. 
A domain  is said to be convex in the horizontal direction (CHD) if its intersection with each horizontal line is connected (or empty).
A function $\varphi$ defined on $D$ is said to be {\em convex in the horizontal direction} (CHD) if $\varphi(D)$ is convex in the horizontal direction.

The following algorithm describes the horizontal shear construction for $f=h+\overline{g}$:

\medskip
\noindent
{\bf Algorithm for horizontal shear construction.}
\begin{enumerate}
\item choosing a conformal mapping $\varphi$ which is convex in horizontal direction;
\item choosing a dilatation $\omega$;
\item computing $h$ and $g$ by solving the system of equations $h-g=:\varphi,~\omega=g'/h'$; 
\item constructing the harmonic mapping $f=h+\overline{g}$.
\end{enumerate}
Clunie and Sheil-Small first introduced this approach in \cite{CS84}, and it was subsequently used by others (see for instance, \cite[Section~3.4, p. 36]{Dur04} and \cite{PQR14}).
Geometrically, a given locally univalent analytic function 
is sheared (i.e. stretched and translated) along parallel lines to produce a 
harmonic mapping onto a domain convex in one direction.

In our discussion, we use this algorithm to take into account harmonic mappings that correspond to the integral transform $C_{\alpha\beta}$ and its rotation with some dilatation depending upon $\alpha$ and $\beta$.
We now recall that Bravo et al. \cite{BHV17} extended the Ahlfors' univalence criteria \cite{Ahl74} to the harmonic case to extend the problem of univalence of $I_\alpha[\varphi]$ to the complex-valued harmonic mappings.
In fact, in \cite{ABHSV20}, a new approach has been initiated to study
the problem of univalence of $I_\alpha[\varphi]$ and $J_\alpha[\varphi]$ 
to the case of harmonic mappings
using the method of shear construction \cite{CS84}. 
The Ces\`aro integral transform and its generalization, however, are not included in either of these two transformations to investigate their univalency in both harmonic and analytical contexts. 
This is the primary justification for our consideration of the integral transform $C_{\alpha\beta}[\varphi]$ to broaden the issues researched in \cite{ABHSV20}.
Indeed, in order to have additional information that incorporates the discoveries from \cite{ABHSV20}, we present a general approach for addressing such issues. Moreover, this generates a number of integral 
transforms of functions that are harmonic and univalent.

%%%%%%%%%%%%%%%%%%%%%%%%%%%%%%%%%%%%%%%%%%%%%%%%%%%%%%%%%%%%%%%%%%
%%%%%%%%%%%%%%%%%%%%%%%%%%%%%%%%%%%%%%%%%%%%%%%%%%%%%%%%%%%%%%%%%
%%%%%%%%%%%%%%%%%%%%% Section 2 %%%%%%%%%%%%%%%%%%%%%%%%%%%%%%%%%%
%%%%%%%%%%%%%%%%%%%%%%%%%%%%%%%%%%%%%%%%%%%%%%%%%%%%%%%%%%%%%%%%%%
%%%%%%%%%%%%%%%%%%%%%%%%%%%%%%%%%%%%%%%%%%%%%%%%%%%%%%%%%%%%%%%%%

\section{Preliminaries}\label{PreliminariesSection}
In this section we collect basic definitions and some well-known results which are used
in the subsequent sections. 
The harmonic Schwarzian and pre-Schwarzian derivatives for sense-preserving harmonic mappings $f=h+\overline{g}$ are investigated in detail by Hern\'andez and Martin in \cite{HM15}. Further applications of harmonic Schwarzian and pre-Schwarzian derivatives for sense-preserving harmonic mappings can be found from  \cite{HM13,HM15-1} and more recently \cite{BHPV22} includes such investigations on logharmonic mappings. Note that
the pre-Schwarzian derivative of a sense-preserving harmonic mapping $f=h+\overline{g}$ 
is defined by
\begin{equation}\label{Eq2.1P1}
P_{f}=\frac{h''}{h'}-\frac{\overline{\omega}\omega'}{1-|\omega|^2}
=\frac{\partial}{\partial z} \log(J_{f}).  
\end{equation}
If $f$ is analytic (i.e. $g\equiv 0$) then $P_{f}=h''/h'$, which is nothing but 
the classical pre-Schwarzian derivative of $f=h$. 
However,  the authors of \cite{HM15} demonstrated that given a sense-preserving harmonic mapping $f$, $P_{f+\overline{af}}=P_{f}$ for $a\in \mathbb{D}$, and 
they established an extension of Becker's criterion of univalence.

\medskip
\noindent
{\bf Lemma~A.}
{\em Let $f=h+\overline{g}$ be a sense-preserving harmonic mapping in the unit disk $\mathbb{D}$ with dilatation $\omega$. If for all $z\in \mathbb{D}$
$$
(1-|z|^{2})|zP_{f}(z)|+\frac{|z\omega^{'}(z)|(1-|z|^{2})}{1-|\omega(z)|^{2}}\leq 1,
$$
then $f$ is univalent. The constant $1$ is the best possible bound.}

Similar types of univalence criteria for harmonic mappings can be found in \cite{ANS16}.
Similar to the case of analytic univalent functions, the notion of pre-Schwarzian derivatives is also used to obtain certain necessary and sufficient conditions for harmonic univalent functions; see \cite{LP19} and Lemma~A respectively.  
Moreover, in 2016, Graf obtained certain bounds of the pre-Schwarzian and Schwarzian 
derivatives in terms of the order of linear and affine invariant families of
sense-preserving harmonic mappings of the unit disk; see \cite{Gra16}.
It is also noteworthy that for the class of uniformly locally univalent harmonic mappings, the authors of \cite{LP18} provided a relationship between its pre-Schwarzian norm and uniformly hyperbolic radius, and also characterized uniformly locally univalent sense-preserving harmonic mappings in multiple ways. 
It is also important to study sufficient conditions for close-to-convexity which also generate more
univalent functions. In this flow, 
the following useful result is quoted from \cite[Theorem~4]{BJJ13}:

\medskip
\noindent
{\bf Lemma~B.}
{\em Let $f=h+\overline{g}$ be a harmonic mapping in $\mathbb{D}$, with $h'(0)\neq 0$ and 
$$
{\rm Re}\,\bigg[1+\frac{zh''(z)}{h'(z)}\bigg]> c
$$
for some $c$ with $-1/2 <c \leq 0$, for all $z\in \mathbb{D}$. If the dilatation $\omega(z)$ satisfies the condition $|\omega(z)|< cos(\pi c)$ for $z\in \mathbb{D}$, then $f$ is close-to-convex in $\mathbb{D}$.}

One can note that $\omega(z)\to 0$ whenever $c \to (-1/2)^+$. Therefore, the case $c=-1/2$ was studied separately by Bharanedhar and Ponnusamy \cite{BP14}. This was initially a conjecture by Mocanu (see \cite[p.~764]{Moc11}) which was later settled in \cite{BL11} for the case $\theta =0$. The authors of \cite{MP16,PK15} further provided some general sufficient conditions for a sense-preserving harmonic mapping to be close-to-convex.

\medskip
Next we deal with certain necessary conditions for univalency of functions 
belonging to linear invariant family (LIF) of analytic functions.
A family $\mathcal{L}$ of normalized locally univalent functions is called LIF, if for any function $\varphi\in \mathcal{L}$, we have 
$$\frac{(\varphi\circ \varphi_{a})(z)-\varphi(a)}{(1-|a|^{2})\varphi'(a)}
\in \mathcal{L},
$$ 
for each automorphism $\varphi_{a}(z)=(z+a)/(1+\overline{a}z)$ of $\mathbb{D}$. 
The concept of LIF was introduced by Pommerenke in 1964 (see \cite{Pom64}) 
and since then it is widely studied in different contexts 
including harmonic mappings of the single and several complex variables, 
see for example \cite{Dur04, GK03}. 
The quantity 
$$
\gamma:=\sup\{|a_2(\varphi)|:\,\varphi(z)\in \mathcal{L}\}
$$
is what determines the {\em order of a family} $\mathcal{L}$,
where $a_2(\varphi)$ is the second Taylor coefficient of $\varphi(z)$. 
Let $\mathcal{L}(\gamma)$ be a linear invariant family of 
analytic functions in $\mathbb{D}$ of order $\gamma$, $\gamma\ge 1$ 
(see \cite{CCP71,Pom64}).
Since $|a_2(\varphi)|\le 2$ for a function $\varphi\in\mathcal{S}$,
it is evident that $\mathcal{S}=\mathcal{L}(2)$.
In connection with the order of LIF, 
the following  lemma, recently showed in \cite[Lemma~3]{ABHSV20}, is used in this
manuscript.

\medskip
\noindent
{\bf Lemma~C.}
{\em  For each univalent function 
$\varphi \in \mathcal{L}(\gamma)$, $1\le \gamma < \infty$, we have
$$
(1-|z|^2)\Big|\frac{z\varphi'(z)}{\varphi(z)}\Big| \leq 2\gamma
$$
for all $z\in \mathbb{D}$.}

\medskip
Next we focus on the concept of stable harmonic univalent functions defined as follows. For this, we frequently use the notation $\mathbb{T}$ to denote the unit 
circle $|z|=1$. 
A sense-preserving harmonic mapping $f = h + \overline{g}$ is called {\em stable
harmonic univalent} (resp. {\em stable harmonic close-to-convex}) in $\mathbb{D}$ 
if all the mappings $f_\lambda= h + \lambda \overline{g}$, $\lambda\in\mathbb{T}$, are univalent (resp. close-to-convex) in $\mathbb{D}$.
We use the notations $\mathcal{SHU}$ and $\mathcal{SHCC}$
to denote the class of stable harmonic univalent functions and the class of stable harmonic close-to-convex functions, respectively. 
Note that the following inclusion relations are well-known: 
$$\mathcal{SHU}\subsetneq \mathcal{S}_{\mathbb{H}},~ \mathcal{SHCC}\subsetneq \mathcal{CC}_{\mathbb{H}},
$$
and also as discussed in \cite{HM13-1} we have 
$$
\mathcal{SHCC} \subsetneq \mathcal{SHU}.
$$
Surprisingly, the authors of \cite{HM13-1} provided the following useful characterization
for a stable harmonic mapping.

\medskip
\noindent
{\bf Lemma~D.}
{\em A function $f=h+\overline{g}$ belongs to $\mathcal{SHU}$ (resp. $\mathcal{SHCC}$) 
if and only if for all $\lambda\in\mathbb{T}$, the analytic function 
$h+\lambda g$ is univalent (resp. close-to-convex).}

%%%%%%%%%%%%%%%%%%%%%%%%%%%%%%%%%%%%%%%%%%%%%%%%%%%%%%%%%%%%%%%%%%
%%%%%%%%%%%%%%%%%%%%%%%%%%%%%%%%%%%%%%%%%%%%%%%%%%%%%%%%%%%%%%%%%
%%%%%%%%%%%%%%%%%%%%% Section 3 %%%%%%%%%%%%%%%%%%%%%%%%%%%%%%%%%%
%%%%%%%%%%%%%%%%%%%%%%%%%%%%%%%%%%%%%%%%%%%%%%%%%%%%%%%%%%%%%%%%%%
%%%%%%%%%%%%%%%%%%%%%%%%%%%%%%%%%%%%%%%%%%%%%%%%%%%%%%%%%%%%%%%%%

\section{Univalence properties}
This section is devoted to the problem of studying the univalence of the integral transform $C_{\alpha\beta}[\varphi]$
whenever $\varphi$ belongs to certain subclasses of the class $\mathcal{S}$.
In addition, we also aim to extend the problem of univalence of $C_{\alpha\beta}[\varphi]$ to the setting of harmonic mappings in the plane. 
For this purpose, we use the method of shear construction as noted in Section~1. 
Throughout this paper we consider $\alpha, \beta \in \mathbb{C}$ unless they are specified.

The first result of this section obtains condition
on $\alpha$ and $\beta$ for which $C_{\alpha\beta}[\varphi]$ is univalent in $\mathbb{D}$
whenever $\varphi\in\mathcal{S}$.

\begin{theorem}\label{thm3.1P1}
If $\varphi \in \mathcal{S}$, then $C_{\alpha \beta}[\varphi]$ is contained in $\mathcal{S}$ 
for $|\alpha|\leq {1}/{[2(2+|\beta|)]}$.
\end{theorem}
\begin{proof}
By the definition of $C_{\alpha\beta}[\varphi]$, the concept of logarithmic derivative
followed by the triangle inequality leads to
$$
(1-|z|^2)\bigg|\frac{z(C_{\alpha \beta}[\varphi])''(z)}
{(C_{\alpha \beta}[\varphi])'(z)}\bigg| 
\leq (1-|z|^2) |\alpha|\left( \bigg|\frac{z\varphi'(z)}{\varphi(z)}-1\bigg|
+\bigg|\frac{\beta z}{1-z}\bigg|\right). 
$$
If $\varphi \in \mathcal{S}$, then Theorem~9 of \cite[p~69]{Goo83} gives that
$\Big|\cfrac{z\varphi'(z)}{\varphi(z)}-1\Big|\leq 2/(1-|z|)$ and so
it follows that
$$
(1-|z|^2)\bigg|\frac{z(C_{\alpha \beta}[\varphi])''(z)}
{(C_{\alpha \beta}[\varphi])'(z)}\bigg| 
\leq |\alpha| \Big(2(1+|z|)+|\beta|(1+|z|)\Big)< 2|\alpha|(2+|\beta|).
$$
Now, by the Becker criterion  \cite{Bec72} for the univalence of an analytic function (see also \cite[Theorem~6.7,~p. 172]{Pom75} and \cite[Theorem~3.3.1,~p. 130]{GK03}), 
$C_{\alpha\beta}[\varphi]$ is univalent in $\mathbb{D}$ provided 
$2|\alpha|(2+|\beta|)\leq 1$ and hence the result follows.
\end{proof}

\begin{remark}
We assume that the bound for $\alpha$ in Theorem~\ref{thm3.1P1} may be improved further, however, for $\alpha, \beta$ satisfying $|\alpha|(2+|\beta|)\geq 2$, we ensure the existence of a function $\varphi \in \mathcal{S}$ such that $C_{\alpha \beta}[\varphi]\notin \mathcal{S}$. This can be seen by considering the Koebe function $\varphi(z)=z/(1-z)^2,\, z\in \mathbb{D}$. Indeed, the corresponding integral transform 
$$
C_{\alpha \beta}[\varphi(z)]=\int_{0}^{z} (1-\zeta)^{-\alpha(2+\beta)}\, d\zeta
$$
is trivially not univalent for $-\alpha(2+\beta)=2$.
\end{remark}

\begin{remark}
For the choice $\beta=0$, Theorem~\ref{thm3.1P1} is equivalent to \cite[Theorem~3]{KM72}.
As a consequence of Theorem~\ref{thm3.1P1}, one may generate a
number of integral transforms that are indeed univalent. 
\end{remark}

Our next purpose is to construct harmonic mappings corresponding to 
the integral transforms $C_{\alpha\beta}$ through shear construction.
From the algorithm described in Section~1, we require to show that
$C_{\alpha\beta}$ is CHD. 

\begin{definition}
	A domain $D \subset \mathbb{C}$ is called {\em convex in the direction} $\theta \,(0\leq \theta <\pi)$ if every line parallel to the line through $0$ and $e^{i\theta}$ has a connected or empty intersection with $D$. A univalent harmonic mapping $f$ in $D$ is said to be {\em convex in the direction} $\theta$ if $f(D)$ is convex in the direction $\theta$. The case $\theta=0$ corresponds to CHD.
\end{definition}

\begin{theorem}\label{thm3.4P1}
	If $\varphi \in \mathcal{S^*(\delta)}$, then $C_{\alpha \beta}[\varphi]$ is convex in one direction in $\mathbb{D}$ for all $\alpha,\beta\geq 0$ satisfying $\alpha(\beta+2(1-\delta))\leq 3$.
\end{theorem}
\begin{proof}
	By the definition of $C_{\alpha\beta}[\varphi]$, we have
	\begin{align*}
		1+{\rm Re}\,\left[\frac{z(C_{\alpha\beta}[\varphi])''(z)}
	{(C_{\alpha\beta}[\varphi])'(z)}\right]
	& = 1+\alpha {\rm Re}\,\left[\frac{z\varphi'(z)}{\varphi(z)}-1+\frac{\beta z}{1-z}\right]\\
	& > 1-\alpha +\alpha \delta-\alpha \beta/2\ge -1/2, 
	\end{align*}
	where the last inequality holds by our assumption $\alpha(\beta+2(1-\delta))\leq 3$.
	Therefore, by using \cite[Theorem 1]{UME52}, one can conclude that $C_{\alpha \beta}[\varphi]$ is convex in one direction in $\mathbb{D}$.
\end{proof}

\medskip
The following result characterizes a function to be CHD.

\medskip
\noindent
{\bf Lemma~E} (\cite[Theorem~1]{RZ76}).
{\em Let $\varphi$ be a non-constant analytic function in $\mathbb{D}$. The function $\varphi$ is CHD if and only if there are numbers $\mu$ and $\nu$, $0\leq \mu <2\pi$ and $0\leq \nu \leq \pi$, such that
$$
{\rm Re}\{e^{i\mu}(1-2ze^{-i\mu}\cos\nu+z^2 e^{-2i\mu})\varphi'(z)\}\geq 0, \quad z\in \mathbb{D}.
$$}

\begin{remark}\label{remark3.5P1}
By Theorem~\ref{thm3.4P1} we learn that 
the operator $C_{\alpha \beta}[\varphi]$ need not be CHD under the same assumptions.
However, for all $\alpha,\beta\geq 0$ satisfying $\alpha(\beta+2(1-\delta))\leq 3$, the rotation $C^\theta_{\alpha \beta}[\varphi](z):=e^{-i\theta}C_{\alpha \beta}[\varphi](e^{i\theta}z)$ of $C_{\alpha \beta}[\varphi](z)$ will be CHD for a suitable choice of $\theta$ whenever $\varphi \in \mathcal{S}^*(\delta)$. In particular, we write $J^{\theta}_{\alpha}[\varphi](z):=e^{-i\theta}J_{\alpha }[\varphi](e^{i\theta}z)$ and $C^{\theta}_{\alpha}[\varphi](z):=e^{-i\theta}C_{\alpha}[\varphi](e^{i\theta}z)$. For instance, we here present an integral operator that is convex in one direction, but not in horizontal direction,
which becomes CHD with a suitable rotation.

For the function $\varphi(z)=z/(1-z^2)$, one can show that by Theorem~\ref{thm3.4P1},
the integral transform $J_{3/2}[\varphi](z)=\int_{0}^{z} (1-\zeta^2)^{-3/2} \, d\zeta$ is convex in one direction. At this moment we do not have any analytical proof for 
$J_{3/2}[\varphi](z)$ to be non-CHD; however the Mathematica graphics tool confirms it  (see Figure~\ref{Fig!AP!}). As a result, we now show that the rotation operator $J^{\pi/4}_{3/2}[\varphi](z)$ is CHD.
	\begin{figure}[H]
		\begin{minipage}[b]{0.45\textwidth}
			\includegraphics[width=5cm,height=5cm]{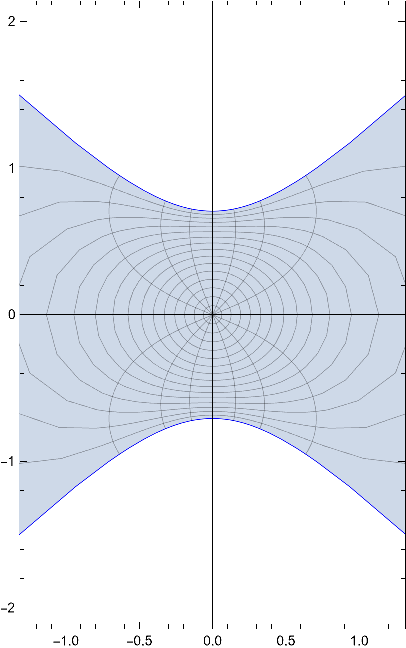}
			\hspace*{1.0cm}  non-CHD $J_{3/2}[\varphi](\mathbb{D})$
		\end{minipage}
		\begin{minipage}[b]{0.4\textwidth}
			\includegraphics[width=6cm,height=5cm]{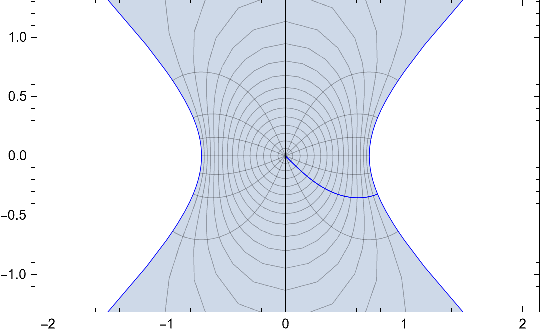}
			\hspace*{2cm} CHD $J^{\pi/4}_{3/2}[\varphi](\mathbb{D})$ 
\end{minipage}
		\caption{The images $J_{3/2}[\varphi](\mathbb{D})$
		and $J^{\pi/4}_{3/2}[\varphi](\mathbb{D})$ for $\varphi(z)=z(1-z^2)^{-1}$}\label{Fig!AP!}
	\end{figure}

Lemma~E, for the choices $\mu=\pi/4,\nu=\pi/2$, leads us in proving 
$$
{\rm Re}\{e^{i\pi/4}(1-iz^2)(J^{\pi/4}_{3/2}[\varphi])'\}={\rm Re}\{(1-iz^2)^{-1/2} \}>0.
$$
This is equivalent to proving $|\arg (1-iz^2)^{-1/2}|<\pi/2$. For this,
consider
$$
k(z)=\int_{0}^{z} (1-i\zeta^2)^{-1}\, d\zeta
$$
and we obtain
$$
1+{\rm Re}\bigg[\frac{z k''(z)}{k'(z)}\bigg]=1+2{\rm Re}\bigg[\frac{iz^2}{1-iz^2}\bigg]>0.
$$
This shows that $k(z)$ is a convex function and therefore,
one can obtain
$$
|\arg (1-iz^2)^{-1/2}|=1/2\cdot|\arg (1-iz^2)^{-1}|<\pi/2.
$$
Therefore, $J^{\pi/4}_{3/2}[\varphi](\mathbb{D})$ is CHD.
\end{remark}

We now define the corresponding harmonic mapping $F^\theta_{\alpha\beta}$ of the integral transform $C^\theta_{\alpha\beta}[\varphi]$ by using the shear construction algorithm as stated in Section~1. Theorem~\ref{thm3.4P1}
and Remark~\ref{remark3.5P1} justify the validity of the following definition:

\begin{definition}\label{Def3.5P1}
Let $\alpha, \beta\ge 0$ and $\alpha(\beta+2(\delta-1))\leq3$. Then we define $F^\theta_{\alpha\beta}(z)=H(z)+\overline{G(z)}$, with the usual normalization $H(0)=G(0)=0, H'(0)=1$ and $G'(0)=0$, as a {\em horizontal shear} of $C^{\theta}_{\alpha\beta}[\varphi](z)=H(z)-G(z)$ having its dilatation $w_{\alpha\beta}(z)=\alpha(1+\beta)w(z)$ for some analytic function $w(z)$ satisfying $|w(z)|<1$.
\end{definition}

Note that one can choose $w$ in such a way that the condition 
$|w_{\alpha\beta}(z)|<1$ is satisfied. In particular, we also use the notations
$\mathcal{F}^\theta_\alpha$ and $\mathcal{G}^\theta_\alpha$ for the 
horizontal shears of $C^\theta_{\alpha}[\varphi]$ and $J^\theta_\alpha[\varphi]$ with their dilatations $w_{\alpha1}$ and $w_{\alpha0}$, respectively.

One can take $F^\theta_{\alpha\beta}=H+\overline{G}$ as a vertical shear of the analytic function $C^\theta_{\alpha\beta}[\varphi]=H+G$ for some $\theta~(0\leq \theta <\pi)$ with the same normalization. However, this small change in the sign produces serious structural difference (see \cite[Section~3.4, p.~40]{Dur04}).

Next, we provide a counterexample to the statement that $F^\theta_{11}\in \mathcal{S}_H$,
a horizontal shear of $C^\theta[\varphi]$, while $\varphi$ ranges over the class $\mathcal{S}^*(\delta),~0\le \delta< 1$. 
This motivates us to study the univalence property of $F^\theta_{\alpha\beta}$ under
certain restrictions on the parameters $\alpha$ and $\beta$. We begin our
investigation with the counterexample followed by the main results.

\begin{example}\label{Ex3.6P1}
For $\lambda\in\mathbb{T}$, consider a locally univalent analytic function 
$\varPhi_{\lambda,\theta}=H+\lambda{G}$ in $\mathbb{D}$. Now $F^\theta_{11}=H+\overline{G}$ is a well defined sense-preserving harmonic mapping, a horizontal shear of $C^\theta[\varphi]=H-G$, with its dilatation $w_{11}=G'/H'$. Adhering to our counterexample, 
we take $\varphi(z)=z/(1-z)^2$ with $\theta=0$ and $w(z)=z/2$. For any $\lambda\in\mathbb{T}$, it is easy to see that the function $\varPhi_{\lambda, 0}=H+\lambda{G}$ satisfies
\begin{equation}\nonumber
	\varPhi'_{\lambda, 0}(z)=H'(z)\cdot [1+\lambda \,w_{11}(z)]=(C^0_{11}[\varphi])'(z) 
	\cdot \frac{1+\lambda z}{1-z}.
\end{equation}
Thus, for all $z\in \mathbb{D}$ and for all $\lambda\in\mathbb{T}$, we compute 
$$
(1-|z|^2)\bigg|\frac{\varPhi''_{\lambda, 0}(z)}{\varPhi'_{\lambda, 0}(z)}\bigg|=(1-|z|^2)\bigg|\frac{4}{1-z}+\frac{\lambda}{1+\lambda z}\bigg|.
$$
By choosing $z=1/2$ and $\lambda=1$, we notice that
$$
\sup_{z\in \mathbb{D}}(1-|z|^2)\bigg|\frac{\varPhi''_{\lambda, 0}(z)}{\varPhi'_{\lambda, 0}(z)}\bigg|\geq \frac{26}{4}>6,
$$
which contradicts the well-known univalence criteria (an immediate consequence of \cite[Theorem 2.4]{Dur83}). Therefore, $\varPhi_{1,0}=H+ G$ is not univalent. It follows by Lemma~D that $F^\theta_{11}\notin \mathcal{S}_{\mathbb{H}}$. 
The graph in relation to the non-univalency of $F^\theta_{11}$ for $\varphi(z)=z/(1-z)^2$ is also shown in Figure~\ref{Fig2P1}.

\begin{figure}[H]
\begin{center}
\includegraphics[width=6.5cm]{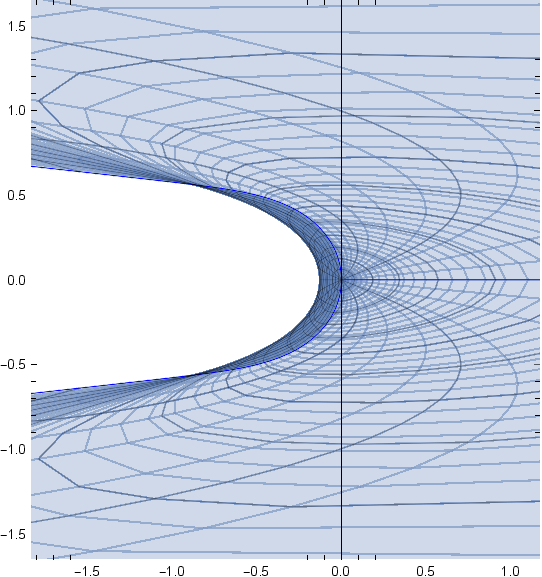}	
\caption{Image of $\mathbb{D}$ under $F_{11}$}\label{Fig2P1}
\end{center}
\end{figure}
\end{example}
 
In what follows, our first main result provides conditions on $\alpha$ and $\beta$ 
for which $F^\theta_{\alpha\beta}$, with its dilatation $w_{\alpha\beta}$, 
is univalent whenever $\varphi$ is a
starlike function of order $\delta,~0\le \delta<1$. 
For this purpose, we use the idea of linearly connected domains.

\begin{theorem}\label{Thm3.7P1}
	Let $\varphi\in\mathcal{S}^*(\delta)$, and $F^\theta_{\alpha\beta}=H+\overline{G}$ be a sense-preserving harmonic mapping in $\mathbb{D}$ with dilatation $w_{\alpha\beta}$. Then for all non-negative parameters $\alpha,\beta$ such that $\alpha(\beta+2(1-\delta))\leq 2$ with 
	$\alpha(1+\beta) \|w\| <1/3$, the corresponding $F^\theta_{\alpha \beta}$ is univalent in $\mathbb{D}$.
\end{theorem}
\begin{proof}
	Let $F^\theta_{\alpha \beta}=H+\overline{G}$
	be a sense-preserving harmonic mapping, which is a horizontal shear of $C^\theta_{\alpha\beta}[\varphi]$.
	We have
\begin{align*}
	1+{\rm Re}\,\left[\frac{z(C^\theta_{\alpha\beta}[\varphi])''(z)}
	{(C^\theta_{\alpha\beta}[\varphi])'(z)}\right]
	& = 1+\alpha {\rm Re}\,\left[\frac{z e^{i\theta}\varphi'(ze^{i\theta})}{\varphi(ze^{i\theta})}-1+\frac{\beta ze^{i\theta}}{1-ze^{i\theta}}\right]\\
	& = 1+\alpha {\rm Re}\,\left[\frac{\zeta \varphi'(\zeta)}{\varphi(\zeta)}-1+\frac{\beta \zeta}{1-\zeta}\right], \quad \zeta=e^{i\theta}z\\
	& > 1-\alpha +\alpha \delta-\alpha \beta/2\ge 0, 
\end{align*}
	where the last inequality holds by our assumption.
	Therefore, $C^\theta_{\alpha\beta}[\varphi]$ is a convex function and so
	$C^\theta_{\alpha\beta}[\varphi](\mathbb{D})$ is a $1$-linearly connected domain; see for instance \cite{CH07,Pom92}. 
Using Lemma~7 of \cite{ABHSV20}, we conclude that $F^\theta_{\alpha \beta}$ is univalent for $\alpha(1+\beta) \|w\|<1/3$.
\end{proof}

\begin{remark}
Since $\mathcal{K}\subset \mathcal{S}^*(1/2)$, Theorem~\ref{Thm3.7P1} is also valid whenever $\varphi$ is a convex function.
\end{remark}

We have a couple of immediate consequences of Theorem~\ref{Thm3.7P1} which 
give the univalency of $\mathcal{G}^\theta_\alpha$ and $\mathcal{F}^\theta_\alpha$.

\begin{corollary}
	Let $\varphi\in\mathcal{K}$, and $\mathcal{G}^\theta_{\alpha}=H+\overline{G}$ be a horizontal shear of $J^\theta_{\alpha}[\varphi]$ with dilatation $w_{\alpha 0}$ in $\mathbb{D}$. Then for all $\alpha\in [0,2]$ with $\alpha\|w\|<1/3$, the mapping $\mathcal{G}^\theta_{\alpha}$ is univalent in $\mathbb{D}$.
\end{corollary}

\begin{corollary}
	Let $\varphi\in\mathcal{K}$, and $\mathcal{F}^\theta_{\alpha}=H+\overline{G}$ be a horizontal shear of $C^\theta_{\alpha}[\varphi]$ with dilatation $w_{\alpha 1}$ in $\mathbb{D}$. 
	Then for all $\alpha\in [0,1]$ with $\alpha\|w\|<1/6$,
	the mapping $\mathcal{F}^\theta_{\alpha}$ is univalent in $\mathbb{D}$.
\end{corollary}

Next we focus on the univalence of $F^\theta_{\alpha\beta}$ in terms of harmonic pre-Schwarzian derivative, where Lemma~A plays a crucial role. For this, a simplified version of the pre-Schwarzian derivative of 
$F^\theta_{\alpha\beta}$ is required. Indeed, by using \eqref{Eq2.1P1}, a direct calculation 
shows that the pre-Schwarzian derivative of $F^\theta_{\alpha\beta}$ is obtained as 
\begin{align}\label{Eq3.1P1}
	 P_{F^\theta_{\alpha \beta}}(z) & =\alpha \Big[\frac{e^{i\theta}\varphi'(ze^{i\theta})}{\varphi(ze^{i\theta})}
	-\frac{1}{z}+\frac{\beta e^{i\theta}}{1-e^{i\theta}z}\\
\nonumber	& \hspace*{0.9cm}+(1+\beta) w'(z)\Big(\frac{1-\overline{\alpha(1+\beta) w(z)}}
	{(1-\alpha(1+\beta) w(z))(1-|\alpha(1+\beta)|^2|w(z)|^2)}\Big)\Big].   
\end{align}

For the sake of convenience, we define the following notation.
Using the classical Schwarz-Pick lemma, we observe that 
\begin{equation}\label{Eq3.2P1}
	\| w^{*} \|= \sup_{z\in \mathbb{D}} \frac{|w^{'}(z)|(1-|z|^{2})}{1-|w|^{2}}\leq 1,
\end{equation}
where $\Vert w^{*} \Vert$ is called the {\em hyperbolic norm} of 
 $w(z)$.

Thus, we have

\begin{theorem}
Let $F^\theta_{\alpha\beta}=H+\overline{G}$ be a sense-preserving harmonic mapping in $\mathbb{D}$ with dilatation $w_{\alpha\beta}$. If $\varphi\in \mathcal{L}(\gamma)$, 
	then 
		\begin{enumerate}
		\item[\bf (i)] for $\beta\geq 1$, $F^\theta_{\alpha \beta}\in \mathcal{S}_{\mathbb{H}}$ for all non-negative values of $\alpha$ satisfying
		\begin{equation}\label{Eq3.3P1}
			\alpha\leq \frac{1}{2\gamma+2\beta+(1+\beta)\,\Vert w^{*} \Vert(1+\Vert w \Vert)]}.
		\end{equation}
		\item[\bf (ii)] for $0\leq \beta<1$, two cases arise. 
		\begin{enumerate}
			\item[\bf (a)] If $(\beta+2(1+\beta)\,\Vert w^{*} \Vert (1+\Vert w \Vert))\le 2(1-\beta)$, then $F^\theta_{\alpha \beta}\in \mathcal{S}_{\mathbb{H}}$ for all non-negative values of $\alpha$ satisfying
			\begin{equation}\label{Eq3.4P1}
				\alpha\leq \frac{4(1-\beta)}{\Big[4(2\gamma+1)(1-\beta)+(\beta+(1+\beta)\,\| w^{*} \|(1+\|w\|))^2+4(1-\beta^2)\| w^{*} \|\Big]}.
			\end{equation}
			\item[\bf (b)] If $(\beta+2(1+\beta)\,\Vert w^{*} \Vert (1+\Vert w \Vert))> 2(1-\beta)$, then $F^\theta_{\alpha \beta}\in \mathcal{S}_{\mathbb{H}}$ 
			for all non-negative values of $\alpha$ satisfying the inequality \eqref{Eq3.3P1}.
		\end{enumerate}
	\end{enumerate}
\end{theorem}
\begin{proof}
Note that, by Lemma~C and \eqref{Eq3.1P1}, for all $z\in \mathbb{D}$ we estimate
\begin{align*}
	(1-|z|^{2})|z P_{F^\theta_{\alpha \beta}}(z)|
	& =(1-|z|^{2})\alpha\left|\frac{z e^{i\theta}\varphi'(z e^{i\theta})}{\varphi(z e^{i\theta})}-1
	+\frac{\beta z e^{i\theta}}{1-z e^{i\theta}}\right.\\
&\hspace*{3cm}	\left.+\frac{z(1+\beta)w'(z)(1-\overline{\alpha(1+\beta) w(z)})}
	{(1-\alpha(1+\beta) w(z))(1-(\alpha(1+\beta))^2|w(z)|^2)}\right|\\
	& \leq \alpha \left[(1-|z|^{2})\Big|\frac{ze^{i\theta}\varphi'(ze^{i\theta})}{\varphi(ze^{i\theta})}\Big|+1-|z|^2+\beta|z|(1+|z|)\right.\\
	& \hspace*{5cm}\left.+\frac{(1-|z|^{2})(1+\beta)|w'(z)||z|}{1-(\alpha(1+\beta))^2|w(z)|^2}\right]\\
	& \leq \alpha \left[2\gamma+1+(\beta-1)|z|^{2}+\left(\beta+(1+\beta)\Vert w^{*} \Vert\Vert w \Vert\right)|z|\right].
\end{align*}
To find the supremum of the right-hand expression, we consider
two cases: 
\begin{enumerate}
	\item[\bf (i)] The case $\beta\geq 1$. 
	
	In this case, the maximum value of the right-hand expression 
	holds trivially for $|z|=1$. 
	This implies that
	$$
	(1-|z|^{2})|z P^\theta_{F_{\alpha \beta}}(z)|
	\leq \alpha [2\gamma+2\beta+(1+\beta)\,\Vert w^{*} \Vert\Vert w \Vert].
	$$
	Thus, we compute
	$$
	(1-|z|^{2})|z P_{F^\theta_{\alpha \beta}}(z)|+\frac{|z w_{\alpha\beta}^{'}(z)|(1-|z|^{2})}{1-|w_{\alpha\beta}(z)|^{2}}\leq \alpha [2\gamma+2\beta+(1+\beta)\,\Vert w^{*} \Vert(1+\Vert w \Vert)].
	$$
	It follows from Lemma~A that $F^\theta_{\alpha \beta}$ is univalent in $\mathbb{D}$, if
	$\alpha$ and $\beta$ satisfy the bound given in \eqref{Eq3.3P1}.
	
	\item[\bf (ii)] The case $\beta< 1$.
	
	Clearly, the maximum value of the right-hand expression
	is attained for 
	$$|z|=\frac{1}{2(1-\beta)}(\beta+(1+\beta)\,\Vert w^{*} \Vert \Vert w \Vert).
	$$
	The supremum quantity is discussed through two subcases, namely,
	\begin{enumerate}
		\item[\bf (a)] The subcase $(\beta+(1+\beta)\,\Vert w^{*} \Vert \Vert w \Vert)
		\le 2(1-\beta)$.
		
		For this case, we have
		$$
		(1-|z|^{2})|z P_{F^\theta_{\alpha \beta}}(z)|
		\leq \frac{\alpha}{4(1-\beta)}\Big[4(2\gamma+1)(1-\beta)+(\beta+(1+\beta)\,\| w^{*} \|(1+\|w\|))^2\Big],
		$$
		and thus,
		\begin{align*}
			&(1-|z|^{2})|z P^\theta_{F_{\alpha \beta}}|+\frac{|z w_{\alpha\beta}^{'}(z)|(1-|z|^{2})}{1-|w_{\alpha\beta}(z)|^{2}}\\
			&\hspace*{0.2cm} \leq \frac{\alpha}{4(1-\beta)}\Big[4(2\gamma+1)(1-\beta)+(\beta+(1+\beta)\,\| w^{*} \|(1+\|w\|))^2\\
			&\hspace*{5cm} +4(1-\beta^2)\| w^{*} \|\Big].
		\end{align*}
		Again using Lemma~A, we conclude that $F^\theta_{\alpha \beta}$ is univalent in $\mathbb{D}$ whenever $\alpha$ satisfies 
		the inequality \eqref{Eq3.4P1}.
		
		\item[\bf (b)] The subcase $(\beta+2(1+\beta)\Vert w^{*} \Vert (1+\Vert w \Vert))
		> 2(1-\beta)$. 
		
		Trivially, the maximum value of the right-hand expression holds for 
		$|z|=1$. Similarly, as an application of Lemma~A, it then follows that
		$F^\theta_{\alpha \beta}$ is univalent in $\mathbb{D}$ whenever $\alpha$ and $\beta$ satisfy 
		the inequality \eqref{Eq3.3P1}.
	\end{enumerate}
\end{enumerate}
This completes the proof.
\end{proof}

The  concludes the univalence properties of $F^\theta_{\alpha\beta}$ for $\varphi$ that belong to specific subclasses of $\mathcal{S}$.

%%%%%%%%%%%%%%%%%%%%%%%%%%%%%%%%%%%%%%%%%%%%%%%%%%%%%%%%%%%%%%%%%%
%%%%%%%%%%%%%%%%%%%%%%%%%%%%%%%%%%%%%%%%%%%%%%%%%%%%%%%%%%%%%%%%%
%%%%%%%%%%%%%%%%%%%%% Section 4 %%%%%%%%%%%%%%%%%%%%%%%%%%%%%%%%%%
%%%%%%%%%%%%%%%%%%%%%%%%%%%%%%%%%%%%%%%%%%%%%%%%%%%%%%%%%%%%%%%%%%
%%%%%%%%%%%%%%%%%%%%%%%%%%%%%%%%%%%%%%%%%%%%%%%%%%%%%%%%%%%%%%%%%

\section{Stable harmonic univalence properties}
This section deals with the stable harmonic univalence properties of $F^\theta_{\alpha\beta}$.
It is evident that $\mathcal{SHU}\subsetneq \mathcal{S}_{\mathbb{H}}$. 
As demonstrated in 
Example~\ref{Ex3.6P1}, $F^\theta_{11}\not\in\mathcal{S}_{\mathbb{H}}$ and 
hence $F^\theta_{11}\not\in\mathcal{SHU}$. 
Therefore, it is also important 
to study the stable harmonic univalence properties of $F^\theta_{\alpha\beta}$.
In fact, our findings show that the conditions on $\alpha$ and $\beta$ alter in the necessary circumstances for $F^\theta_{\alpha\beta}\in\mathcal{SHU}$, just as they appeared
in the case of $F^\theta_{\alpha\beta}\in\mathcal{S}_{\mathbb{H}}$. 

Our first result determines conditions
on $\alpha$ and $\beta$ for which $F^\theta_{\alpha\beta}\in\mathcal{SHU}$ whenever $\varphi\in\mathcal{S}^*(\delta)$. 

\begin{theorem}\label{Thm4.1P1}
Let $F^\theta_{\alpha\beta}$ be a sense-preserving harmonic mapping in $\mathbb{D}$ with dilatation $w_{\alpha\beta}$. If $\varphi\in\mathcal{S}^*(\delta)$ 
then $F^\theta_{\alpha \beta}\in \mathcal{SHU}$ for all non-negative $\alpha,\beta$ satisfying 
\begin{equation}\label{Eq4.1P1}
	\alpha\leq \frac{1}{2\Big(2+\beta+(1+\beta)\,\Vert w^{*} \Vert 
		(1+\Vert w \Vert)\Big)}.   
\end{equation}
\end{theorem}
\begin{proof}
Since $\varphi\in\mathcal{S}^*(\delta)$, we have $\varphi(0)=0$ which justifies the 
local univalence of $C^\theta_{\alpha \beta}[\varphi]$ and so $F^\theta_{\alpha \beta}=H+\overline{G}$ is well-defined. 
It is easy to see that for any $\lambda\in\mathbb{T}$, the function $\varPhi_{\lambda, \theta}=H+\lambda G$ satisfies
\begin{equation}\label{Eq4.2P1}
	\varPhi'_{\lambda, \theta}(z)=H'(z)\cdot [1+\lambda w_{\alpha \beta}(z)]=(C^\theta_{\alpha \beta}[\varphi])'(z) 
	\cdot \frac{1+\lambda \alpha(1+\beta) \,w(z)}{1-\alpha(1+\beta) \,w(z)}.
\end{equation}
Hence, for all $z\in \mathbb{D}$, we have
\begin{align}\label{Eq4.3P1}
	(1-|z|^{2})\bigg|\frac{z\,\varPhi''_{\lambda, \theta}(z)}{\varPhi'_{\lambda, \theta}(z)}\bigg|
=(1-|z|^{2})\alpha \bigg|\frac{z e^{i\theta}\varphi'(z e^{i\theta})}{\varphi(ze^{i\theta})}-1+\frac{\beta ze^{i\theta}}
	{1-ze^{i\theta}}
	&+\frac{\lambda(1+\beta) z \,w'(z)}{1+\lambda (1+\beta)\alpha \,w(z)}\\
&\nonumber	+\frac{z (1+\beta)\,w'(z)}{1-\alpha(1+\beta) \,w(z)}\bigg|.   
\end{align}
Since $w(z)$ is a self-map of $\mathbb{D}$ and $|z\varphi'(z)/\varphi(z)-1| 
\leq 2/(1-|z|)$, by the classical distortion theorem for $\mathcal{S}$ and \eqref{Eq3.2P1}, 
we find
\begin{align*}
	(1-|z|^{2})\bigg|z\frac{\varPhi''_{\lambda, \theta}(z)}{\varPhi'_{\lambda, \theta}(z)}\bigg|
	& \leq \alpha\Big(2(1+|z|)+\beta(1+|z|)+2(1+\beta)\,\Vert w^{*} \Vert 
	(1+\Vert w \Vert)|z|\Big)\\
	& \le \alpha\Big(4+2\beta+2(1+\beta)\,\Vert w^{*} \Vert 
	(1+\Vert w \Vert)\Big).
\end{align*}
It follows that $\varPhi_{\lambda, \theta}$ satisfies the Becker 
univalence criterion for all $\lambda \in \mathbb{T}$ 
(see \cite{Bec72} and also \cite[Theorem~3.3.1,~p. 130]{GK03}), 
whenever $\alpha,\beta$ are related by \eqref{Eq4.1P1}.
Therefore, by Lemma~D, $F^\theta_{\alpha \beta}$ belongs to the class $\mathcal{SHU}$
under the restriction given by \eqref{Eq4.1P1}.
\end{proof}

For the choice $\beta=1$, Theorem~\ref{Thm4.1P1} produces 
the stable harmonic univalence of $\mathcal{F}^\theta_\alpha$ as follows:  

\begin{corollary}
Let $\mathcal{F}^\theta_{\alpha}$ be a horizontal shear of $C^\theta_{\alpha}[\varphi]$ with dilatation $w_{\alpha 1}$ in $\mathbb{D}$.  
If $\varphi \in \mathcal{S}^*(\delta)$, then $\mathcal{F}^\theta_{\alpha}\in \mathcal{SHU}$ 
for all non-negative $\alpha$ satisfying 
$$
\alpha\leq \frac{1}{2\Big(3+2\Vert w^{*} \Vert (1+\Vert w \Vert)\Big)}.
$$
\end{corollary}

Similarly, for the choice $\beta=0$, Theorem~\ref{Thm4.1P1} produces the well-known
fact about the stable harmonic univalency of $\mathcal{G}^\theta_\alpha$ (see
\cite[Theorem~2]{ABHSV20}), for $\alpha\geq 0$, as follows:  

\begin{corollary}
Let $\mathcal{G}^\theta_{\alpha}$ be a horizontal shear of $J^\theta_{\alpha}[\varphi]$ with dilatation $w_{\alpha 0}$ in $\mathbb{D}$. If $\varphi \in \mathcal{S}^*(\delta)$, 
then $\mathcal{G}^\theta_{\alpha}\in \mathcal{SHU}$ for all non-negative $\alpha$ satisfying 
$$
\alpha\leq \frac{1}{2\Big(2+\Vert w^{*} \Vert (1+\Vert w \Vert)\Big)}.
$$
\end{corollary}

Next we discuss the stable harmonic univalence of $F^\theta_{\alpha\beta}$
when $\varphi$ belongs to a class of linear invariant family. 

\begin{theorem}
Let $\alpha\ge 0$ and
$F^\theta_{\alpha\beta}$ be a sense-preserving harmonic mapping in $\mathbb{D}$ with dilatation $w_{\alpha\beta}$. 
If $\varphi\in\mathcal{L}(\gamma)$,  $1\le \gamma<\infty$, then we have
\begin{enumerate}
	\item[\bf (i)] For $\beta\geq 1$, $F^\theta_{\alpha \beta}\in \mathcal{SHU}$ for all values of $\alpha$ satisfying
	\begin{equation}\label{Eq4.4P1}
		\alpha\leq \frac{1}{2\Big(\gamma+\beta+(1+\beta)\,\|w^{*}\| (1+\|w\|)\Big)}.
	\end{equation}
	\item[\bf (ii)] For $0\le \beta<1$, two cases arise. 
	\begin{enumerate}
		\item[\bf (a)] If $\beta+2(1+\beta)\,\|w^{*}\| (1+\|w\|)\le 
		2(1-\beta)$, then $F^\theta_{\alpha \beta}\in \mathcal{SHU}$ for all values of $\alpha$ satisfying
		\begin{equation}\label{Eq4.5P1}
			\alpha\leq \frac{4(1-\beta)}{4(2\gamma+1)(1-\beta)+(\beta+2(1+\beta)\,\| w^{*} \|(1+\|w\|))^2}.
		\end{equation}
		\item[\bf (b)] If $\beta+2(1+\beta)\,\|w^{*} \|(1+\|w\|)
		> 2(1-\beta)$, then $F^\theta_{\alpha \beta}\in \mathcal{SHU}$ 
		for values of $\alpha$ satisfying the inequality \eqref{Eq4.4P1}.
	\end{enumerate}
\end{enumerate}
\end{theorem}
\begin{proof}
Using Lemma~C and \eqref{Eq4.3P1}, we get
\begin{align*}
	(1-|z|^{2})\bigg|\frac{z\varPhi''_{\lambda, \theta}(z)}{\varPhi'_{\lambda, \theta}(z)}\bigg| & \leq \alpha\Big(2\gamma+1-|z|^{2}+\beta(1+|z|)|z|+2(1+\beta)\,\|w^{*}\| (1+\|w\|)|z|\Big)\\ 
	& = \alpha\Big(2\gamma+1+(\beta-1)|z|^{2}+(\beta+2(1+\beta)\,\|w^{*}\|(1+\|w\|))|z|\Big).
\end{align*}
To find the supremum of the right-hand expression, we consider
two cases: 
\begin{enumerate}
	\item[\bf (i)] The case $\beta\geq 1$. 
	
	In this case, the maximum value of the right-hand expression 
	holds trivially for $|z|=1$. Therefore, $\varPhi_{\lambda, \theta}$ 
	satisfies the Becker univalence criterion for all $\lambda \in \mathbb{T}$ whenever $\alpha$ satisfies the inequality \eqref{Eq4.4P1}.
	
	\item[\bf (ii)] The case $0\le \beta< 1$.
	
	Clearly, the maximum value of the right-hand expression
	is attained for 
	$$|z|=\frac{1}{2(1-\beta)}(\beta+2(1+\beta)\,\|w^{*}\|(1+\|w \|)).
	$$
	The supremum quantity is discussed through two subcases, namely,
	\begin{enumerate}
		\item[\bf (a)] The subcase $\beta+2(1+\beta)\,\|w^{*}\|(1+\|w\|)\le 
		2(1-\beta)$.
		
		In this case, $\varPhi_{\lambda, \theta}$ satisfies the Becker univalence 
		criterion for all $\lambda \in \mathbb{T}$ when 
		$\alpha$ satisfies the inequality \eqref{Eq4.5P1}. 
		
		\item[\bf (b)] The subcase $\beta+2(1+\beta)\,\|w^{*}\|(1+\|w\|)
		> 2(1-\beta)$. 
		
		Trivially, the maximum value of the right-hand expression holds for 
		$|z|=1$. It follows that $\varPhi_{\lambda, \theta}$ satisfies the Becker univalence
		criterion for all $\lambda \in \mathbb{T}$ whenever $\alpha$ satisfies the inequality \eqref{Eq4.4P1}.
	\end{enumerate}
\end{enumerate}
This completes the proof.
\end{proof}

Until this point, whenever $\varphi$ is univalent, we have dealt with the stable harmonic univalence properties of $F^\theta_{\alpha\beta}$ . 
The features of $F^\theta_{\alpha\beta}$ that are close-to-convex are examined in the next section whenever $\varphi$ belongs to certain subclasses of $\mathcal{S}$. Additionally, we offer bounds on $\alpha$ and $\beta$ under which $F^\theta_{\alpha\beta}$ is close-to-convex.

%%%%%%%%%%%%%%%%%%%%%%%%%%%%%%%%%%%%%%%%%%%%%%%%%%%%%%%%%%%%%%%%%%
%%%%%%%%%%%%%%%%%%%%%%%%%%%%%%%%%%%%%%%%%%%%%%%%%%%%%%%%%%%%%%%%%
%%%%%%%%%%%%%%%%%%%%% Section 5 %%%%%%%%%%%%%%%%%%%%%%%%%%%%%%%%%%
%%%%%%%%%%%%%%%%%%%%%%%%%%%%%%%%%%%%%%%%%%%%%%%%%%%%%%%%%%%%%%%%%%
%%%%%%%%%%%%%%%%%%%%%%%%%%%%%%%%%%%%%%%%%%%%%%%%%%%%%%%%%%%%%%%%%

\section{Close-to-Convexity properties}
Recall that $\mathcal{CC}_{\mathbb{H}}\subsetneq \mathcal{S}_{\mathbb{H}}$. 
The function $F^\theta_{11}$ does not belong to $\mathcal{S}_{\mathbb{H}}$ as seen in Example~\ref{Ex3.6P1} and subsequently $F^\theta_{11}\not\in\mathcal{CC}_{\mathbb{H}}$.
The phenomena of close-to-convexity of $F^\theta_{\alpha\beta}$ must therefore be studied.
In fact, our results show that the conditions on $\alpha$ and $\beta$ alter in the necessary circumstances for $F^\theta_{\alpha\beta}\in\mathcal{CC}_{\mathbb{H}}$, just as they appeared
in the case of $F^\theta_{\alpha\beta}\in\mathcal{S}_{\mathbb{H}}$. 

Our first result in this section provides the conditions on $\alpha$ and $\beta$ 
under which $F^\theta_{\alpha\beta}\in \mathcal{CC}_{\mathbb{H}}$ whenever $\varphi\in \mathcal{S}^*(\delta)$. 

\begin{theorem}\label{Thm5.1P1}
Let $F^\theta_{\alpha\beta}=H+\overline{G}$ be a sense-preserving harmonic mapping in $\mathbb{D}$ with dilatation $w_{\alpha\beta}$. If $\varphi\in \mathcal{S^*}(\delta)$ and $w(z)=\cos(\pi c) z/2$, for some $c, -1/2< c<0$, then for all non-negative parameters $\alpha,\beta$ satisfying $\alpha(1+\beta)\leq 1$ with $\alpha(2(1-\delta)+\beta)\le -2c$, we have $F^\theta_{\alpha\beta}\in \mathcal{CC}_{\mathbb{H}}.$ 
\end{theorem}
\begin{proof}
Since $\varphi\in \mathcal{S^*(\delta)}$, by Definition~\ref{Def3.5P1}, the harmonic mapping $F^\theta_{\alpha\beta}=H+\overline{G}$ is well-defined.
	Clearly, for the given choice of $w(z)$, we have
	$$
    |w_{\alpha\beta}(z)|=\alpha(1+\beta) |w(z)|<\frac{\cos(\pi |c|)}{2}<\cos(\pi |c|).
	$$
	Since $C^\theta_{\alpha\beta}[\varphi]=H-G$ satisfies $(C^\theta_{\alpha\beta}[\varphi])'(z)=H'(z)(1-w_{\alpha\beta}(z))$, for all $z\in \mathbb{D}$, 
	it follows that
\begin{align*}
1+{\rm Re} \bigg[\frac{z H''(z)}{H'(z)}\bigg] & =1+\alpha{\rm Re}\bigg[\frac{ze^{i\theta}\varphi'(ze^{i\theta})}{\varphi(z e^{i\theta})}-1+\frac{\beta ze^{i\theta}}{1-ze^{i\theta}}\bigg]+{\rm Re}\bigg[\frac{z w_{\alpha\beta}'(z)}{1-w_{\alpha\beta}(z)}\bigg]\\
		& =1+\alpha{\rm Re}\bigg[\frac{\zeta\varphi'(\zeta)}{\varphi(\zeta)}-1+\frac{\beta \zeta}{1-\zeta}\bigg]-{\rm Re}\bigg[\frac{-\alpha(1+\beta)zw'(z)}{1-\alpha(1+\beta)w(z)}\bigg]\\
		& > 1+ \alpha\delta-\alpha-\alpha\beta/2-1
		\geq c,	
	\end{align*}
with $\zeta=e^{i\theta} z$, where the last inequality follows since $\alpha(2(1-\delta)+\beta)\le -2c$.
	Therefore according to Lemma~B, $F^\theta_{\alpha\beta}$ is a close-to-convex mapping.
\end{proof}

Recall that the connection $\mathcal{K}\subset \mathcal{S}^*(1/2)$ is valid. 
Therefore,
Theorem~\ref{Thm5.1P1} offers the following univalence close-to-convexity of $\mathcal{G}^\theta_\alpha$, if $\beta=0$ is chosen.

\begin{corollary}
Let $\mathcal{G}^\theta_{\alpha}$ be a sense-preserving harmonic mapping in $\mathbb{D}$ with dilatation $w_{\alpha0}$. If $\varphi\in \mathcal{K}$ and $w(z)=\cos(\pi c) z/2$, for some $c, -1/2< c<0$, then for all $\alpha \in [0,-2c]$, the mapping  $\mathcal{G}^\theta_{\alpha}\in\mathcal{CC}_{\mathbb{H}}$.
\end{corollary}

In the similar fashion, if one chooses $\beta=1$ in Theorem~\ref{Thm5.1P1}, then the close-to-convexity of $\mathcal{F}^\theta_\alpha$ follows.

\begin{corollary}
Let $\mathcal{F}^\theta_{\alpha}$ be a sense-preserving harmonic mapping in $\mathbb{D}$ with dilatation $w_{\alpha1}$. If $\varphi\in \mathcal{K}$ and $w(z)=\cos(\pi c) z/2$, for some $c, -1/2< c<0$, then for all $\alpha \in [0,-c]$, the mapping  $\mathcal{F}^\theta_{\alpha}\in\mathcal{CC}_{\mathbb{H}}$. 
\end{corollary}

The stable harmonic close-to-convexity of $F^\theta_{\alpha\beta}$, whenever $\varphi\in\mathcal{S}^*(\delta)$, 
is the subject of our next major finding. However, this is dependent on the next elementary lemma. In the remaining section we choose $w(z)=z/2$.

\begin{lemma}\label{Lem5.4P1}
Let $F^\theta_{\alpha\beta}$ be a sense-preserving harmonic mapping in $\mathbb{D}$ with dilatation $w_{\alpha\beta}$.
	Then for all $\lambda\in\mathbb{T}$ and for all non-negative $\alpha,\beta$ with $\alpha(1+\beta)\leq 1$, we have 
	$$
	\bigg| \arg \bigg(\frac{2+\lambda \alpha(1+\beta) z}{2-\alpha(1+\beta) z}\bigg)\bigg| \leq 2\arcsin(r\alpha(1+\beta)),
	$$ where $r=|z|<1$.
\end{lemma}
\begin{proof}
	For any $\lambda\in\mathbb{T}$, the relation \eqref{Eq4.2P1} suggests us to consider
	the integral
	$$
	I(z)=\int_{0}^{z}\frac{	\varPhi'_{\lambda, \theta}(\zeta)}{(C^\theta_{\alpha \beta}[\varphi])'(\zeta)}\, d\zeta
	=\int_{0}^{z}\frac{2+\lambda \alpha(1+\beta) \zeta}{2-\alpha(1+\beta) \zeta}\, d\zeta.
	$$
	Whence for all $z\in \mathbb{D}$, the logarithmic derivative of $I'(z)$ leads to
	\begin{equation}\label{Eq5.1P1}
		1+{\rm Re} \bigg[\frac{z I''(z)}{I'(z)}\bigg]=1+{\rm Re}\bigg[\frac{z \lambda \alpha(1+\beta)}{2+\lambda \alpha(1+\beta) z}\bigg]-{\rm Re}\bigg[\frac{-z\alpha(1+\beta)}{2-\alpha(1+\beta) z}\bigg].
	\end{equation}
It follows that
$$
{\rm Re}\bigg[\frac{z \lambda \alpha(1+\beta)}{2+\lambda \alpha(1+\beta) z}\bigg]=\frac{\partial}{\partial \theta}\{\arg (2+\lambda \alpha(1+\beta) re^{i\theta})\}, \quad z=re^{i\theta}.
$$
Geometrically, the function $2+\lambda \alpha(1+\beta)z$ being a M\"obius transformation, it maps each circle $|z|=r<1$ onto another circle. It thus follows that $\arg (2+\lambda \alpha(1+\beta)z)$ increases as $z$ moves around the circle $|z|=r$ in the positive sense.
That is,
$$
\frac{\partial}{\partial \theta}\{\arg (2+\lambda \alpha(1+\beta) re^{i\theta})\}>0, \quad z=re^{i\theta}.
$$
Equivalently, on the one hand, we have
$$
{\rm Re}\bigg[\frac{z \lambda \alpha(1+\beta)}{2+\lambda \alpha(1+\beta)z}\bigg]>0.
$$
On the other hand, one can easily see that
	$$
	{\rm Re}\bigg[\frac{-z\alpha(1+\beta)}{2-\alpha(1+\beta)z}\bigg]\le \frac{1}{2-|z|}<1.
	$$
	Thus, by \eqref{Eq5.1P1}, we obtain 
	$$
	1+{\rm Re} \bigg[\frac{z I''(z)}{I'(z)}\bigg]>0,
	$$
	leading to the convexity of $I(z)$ in $\mathbb{D}$. 
	Now, the rotation theorem for convex functions \cite[Page 103]{Dur83}, yields
	$$
	|\arg (I'(z))|=\bigg| \arg \bigg(\frac{2+\lambda \alpha(1+\beta)z}{2-\alpha (1+\beta)z}\bigg)\bigg| \leq 2\arcsin(r\alpha(1+\beta)), \quad |z|=r<1,
	$$
	completing the proof.
\end{proof}

\begin{theorem}\label{Thm5.5P1}
Let $F^\theta_{\alpha\beta}$ be a sense-preserving harmonic mapping in $\mathbb{D}$ with its dilatation $w_{\alpha\beta}$. 
	If $\varphi\in\mathcal{S}^*(\delta)$ and $\alpha \in [0, 1/(1+\beta)\sqrt{2}]$, then $F^\theta_{\alpha \beta}\in \mathcal{SHCC}$. 
\end{theorem}
\begin{proof}
	Let $\lambda \in \mathbb{T}$ be arbitrary. Consider 
	$\varPhi_{\lambda, \theta}$ as defined
	in the proof of Theorem \ref{Thm4.1P1}. 
	
	For $0\leq t_{2}-t_{1}\leq 2\pi$ and $z=re^{it}$, 
	we first compute 
	\begin{align*}
		\int_{t_1}^{t_2}  {\rm Re}\,\bigg [1+\frac{z \varPhi''_{\lambda, \theta}(z)}{\varPhi'_{\lambda, \theta}(z)}\bigg] \, dt
		& = \int_{t_1}^{t_2} \bigg(1+ {\rm Re}\,
		\bigg[\frac{\alpha z e^{i\theta}\varphi'(z e^{i\theta})}{\varphi(ze^{i\theta})}-\alpha+\frac{\alpha \beta ze^{i\theta}}{1-ze^{i\theta}}\\
		&\hspace*{2cm} +\frac{z\lambda \alpha(1+\beta)}{2+\lambda \alpha (1+\beta)z}+\frac{z\alpha(1+\beta)}{2-\alpha(1+\beta)z} \bigg]\,\bigg)\, dt\\ 
		& > \Big(1+(\delta-1)\alpha-\frac{\alpha \beta}{2}\Big)(t_{2}-t_{1})\\
		&	\hspace*{1cm}	+{\rm arg}\, \bigg(\frac{2+\lambda \alpha(1+\beta) re^{it_2}}{2+\lambda \alpha(1+\beta) re^{it_1}}\cdot \frac{2-\alpha (1+\beta)re^{it_1}}{2-\alpha (1+\beta)re^{it_2}}
		\bigg).
	\end{align*}
	Since $t_2-t_1\ge 0$, it follows that
	\begin{align*}
		\int_{t_1}^{t_2} 
		{\rm Re}\,\bigg [1+\frac{z\varPhi''_{\lambda, \theta}(z)}{\varPhi'_{\lambda, \theta}(z)}\bigg] \, dt 
		& > {\rm arg}\, \bigg(\frac{2+\lambda \alpha (1+\beta) re^{it_2}}{2-\alpha(1+\beta) re^{it_2}}\bigg)\\
	&\hspace*{2cm}	+{\rm arg}\, \bigg( \frac{2-\alpha (1+\beta) re^{it_1}}{2+\lambda\alpha (1+\beta) re^{it_1}}
		\bigg)\\
		& \geq -4\,\arcsin(r\alpha(1+\beta))>-4\,\arcsin(\alpha(1+\beta)),
	\end{align*}
	where the second inequality holds by Lemma~\ref{Lem5.4P1}.
	Note that, if $\arcsin(\alpha(1+\beta))\leq \pi/4$, or equivalently $0\leq \alpha(1+\beta)\leq 1/\sqrt{2}$, immediately give us
	$$
	\int_{t_1}^{t_2}  {\rm Re} \Big \lbrace 1+z\frac{\varPhi''_{\lambda, \theta}(z)}{\varPhi'_{\lambda, \theta}(z)}\Big \rbrace\, dt > -\pi.
	$$
	Hence $\varPhi_{\lambda, \theta}$ is a close-to-convex mapping in the unit disk. This completes the proof.
\end{proof}

As we recall $\mathcal{K}\subsetneq\mathcal{S}^*(1/2)$, 
Theorem~\ref{Thm5.5P1} provides the following immediate consequences,
respectively for $\beta=0$ and $\beta=1$:

\begin{corollary}\label{Cor5.6P1}
Let $\mathcal{G}^\theta_{\alpha}$ be a horizontal shear of $J^\theta_{\alpha}[\varphi]$ with dilatation $w_{\alpha 0}$ in $\mathbb{D}$. If $\varphi\in\mathcal{K}$, then for all	$\alpha\in[0, 1/\sqrt{2}]$, we have   $\mathcal{G}^\theta_{\alpha}\in\mathcal{SHCC}$.
\end{corollary}

\begin{corollary}\label{Cor5.7P1}
Let $\mathcal{F}^\theta_{\alpha}$ be a horizontal shear of $C^\theta_{\alpha}[\varphi]$ with dilatation $w_{\alpha 1}$ in $\mathbb{D}$. If $\varphi\in\mathcal{K}$, then for all	$\alpha\in [0, 1/2\sqrt{2}]$, we have $\mathcal{F}^\theta_{\alpha}\in\mathcal{SHCC}$.
\end{corollary}

%%%%%%%%%%%%%%%%%%%%%%%%%%%%%%%%%%%%%%%%%%%%%%%%%%%%%%%%%%%%%%%%%%
%%%%%%%%%%%%%%%%%%%%%%%%%%%%%%%%%%%%%%%%%%%%%%%%%%%%%%%%%%%%%%%%%
%%%%%%%%%%%%%%%%%%%%% Section 6 %%%%%%%%%%%%%%%%%%%%%%%%%%%%%%%%%%
%%%%%%%%%%%%%%%%%%%%%%%%%%%%%%%%%%%%%%%%%%%%%%%%%%%%%%%%%%%%%%%%%%
%%%%%%%%%%%%%%%%%%%%%%%%%%%%%%%%%%%%%%%%%%%%%%%%%%%%%%%%%%%%%%%%%

\section{Applications}\label{Sec6P1}
As an application to Theorems~\ref{Thm3.7P1}, \ref{Thm4.1P1}, \ref{Thm5.1P1} and \ref{Thm5.5P1}, 
in this section, we construct harmonic univalent mappings $F^\theta_{\alpha\beta}$
for certain elementary choices of $\varphi$ and their dilatations.

\begin{example}\label{Ex6.1P1}
Consider a non-constant analytic function $w(z)=-z$.
We choose $\varphi(z)=z/(1-z)\in \mathcal{S^{*}}$ 
in the definition of $C_{\alpha\beta}[\varphi]$ and obtain
\begin{equation}\label{Eq6.1P1}
C^\theta_{\alpha\beta}[\varphi](z)=e^{-i\theta}\int_{0}^{z} \frac{1}{(1-e^{i\theta}\zeta)^{\alpha (\beta+1)}}\, d\zeta=e^{-2i\theta}\bigg(\frac{1-(1-e^{i\theta}z)^{1-\alpha(1+\beta)}}{1-\alpha(1+\beta)}\bigg).
\end{equation}
By Remark~\ref{remark3.5P1}, first we note that $C^\theta_{\alpha\beta}[\varphi]$
is CHD for some $\theta~ (0\leq \theta <\pi)$ and following the construction given in Section~3, we can construct $F^\theta_{\alpha\beta}=H+\overline{G}$, a horizontal shear of $C^\theta_{\alpha\beta}[\varphi]=H-G$ 
defined by \eqref{Eq6.1P1} with dilatation $w_{\alpha \beta}=-\alpha(1+\beta)z$. 
It leads to
$$H-G=C^\theta_{\alpha\beta}[\varphi] ~~\mbox{ and }~~ \frac{G'(z)}{H'(z)}=-\alpha(1+\beta) z.
$$
The second equation along with the differentiation of the first equation produces
a system of equations in $H'$ and $G'$. An elementary calculation thus yields
$$
H(z)=e^{-i\theta}\int_{0}^{z} \frac{1}{(1-e^{i\theta}\zeta)^{\alpha(1+\beta)}(1+\alpha(1+\beta)\zeta)}\, d\zeta
$$ 
and
$$
G(z)=e^{-i\theta}\int_{0}^{z} \frac{-\alpha(1+\beta)\zeta}{(1-e^{i\theta}\zeta)^{\alpha(1+\beta)}(1+\alpha(1+\beta)\zeta)}\, d\zeta
$$
under the usual normalization $H(0)=G(0)=0$. Therefore, the harmonic mapping 
\begin{align}\label{Eq6.2P1}
F^\theta_{\alpha\beta}(z)=H(z)+\overline{G(z)} 
\end{align}
maps the unit disk onto a domain convex in the horizontal direction.
	
By using Theorem \ref{Thm3.7P1}, the mapping 
$F^\theta_{\alpha\beta}$ given by \eqref{Eq6.2P1} is univalent for all non-negative $\alpha,\beta$ satisfying 
\begin{equation}\label{Eq6.3P1}
\alpha(1+\beta)<1/3,
\end{equation}
since in this case we have $\|w\|=1$.
The first image of Figure~\ref{Fig3P1} demonstrates the univalence of $F^\theta_{\alpha\beta}$
for $\alpha,\beta$ satisfying \eqref{Eq6.3P1}, whereas the second image shows that there are 
non-univalent functions $F^\theta_{\alpha\beta}$ for $\alpha,\beta$ not satisfying \eqref{Eq6.3P1}. 

\begin{figure}[H]
\begin{minipage}[b]{0.55\textwidth}
\includegraphics[width=6.5cm]{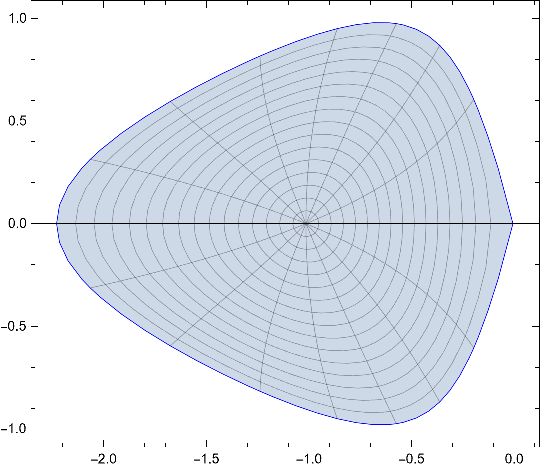}
\hspace*{1.3cm} For $\alpha=1/5$ and $\beta=1/2$
\end{minipage}
\begin{minipage}[b]{0.3\textwidth}
\includegraphics[width=5.5cm, height=5.8cm]{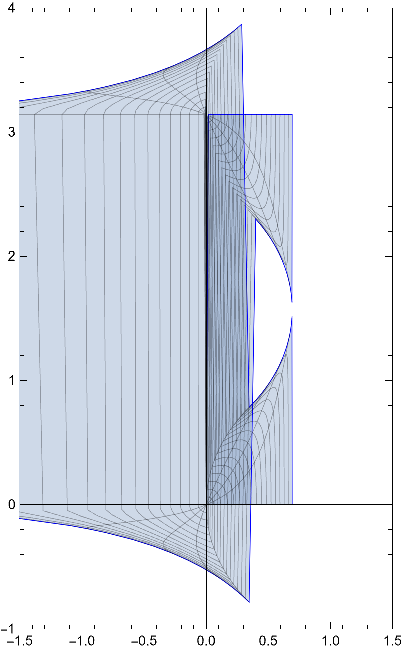}
\hspace*{0.3cm} For $\alpha=1/2$ and $\beta=1$
\end{minipage}
\caption{The image domains $F^\theta_{\alpha\beta}(\mathbb{D})$ for the above choices of $\alpha,\beta$.}\label{Fig3P1}
\end{figure}
\end{example}

\begin{remark}
As a consequence of Example~\ref{Ex6.1P1}, 
there are some $\alpha,\beta$ with 
$1/3\le \alpha(1+\beta)\le 1$ for which $F^\theta_{\alpha\beta}$ is locally univalent but not univalent. Moreover, similar remark also applies to the subsequent examples.
\end{remark}

\begin{example}
Consider $\varphi(z)=z$ and the function $w(z)=(2z+1)/(2+z)$. 
For this $\varphi$, the definition of $C^\theta_{\alpha\beta}[\varphi]$ is equivalent to 
\begin{equation}\label{Eq6.4P1}
C^\theta_{\alpha\beta}[\varphi](z)=e^{-i\theta}\int_{0}^{z} \frac{1}{(1-e^{i\theta}\zeta)^{\alpha \beta}}\, d\zeta=e^{-2i\theta}\bigg(\frac{1-(1-e^{i\theta}z)^{1-\alpha \beta}}{1-\alpha \beta}\bigg).
\end{equation}
Similar to the explanations used in Example~\ref{Ex6.1P1},
we can construct $F^\theta_{\alpha\beta}=H+\overline{G}$ and it generates
$$
H-G=C^\theta_{\alpha\beta}[\varphi] ~~\mbox{ and }~~ \frac{G'(z)}{H'(z)}= \alpha(1+\beta)\frac{2z+1}{2+z}.
$$
By solving these two equations, we obtain
$$
H(z)=e^{-i\theta}\int_{0}^{z} \frac{2+\zeta}{(1-e^{i\theta}\zeta)^{\alpha \beta}[(1-2\alpha(1+\beta))\zeta+2-\alpha(1+\beta)]}\, d\zeta
$$
and
$$
G(z)=e^{-i\theta}\int_{0}^{z}\frac{\alpha(1+\beta)(2\zeta+1)}{(1-e^{i\theta}\zeta)^{\alpha \beta}[(1-2\alpha(1+\beta))\zeta+2-\alpha(1+\beta)]}\, d\zeta
$$
under the standard normalization $H(0)=G(0)=0$. 
Therefore, the harmonic mapping 
\begin{align}\label{Eq6.5P1}
F^\theta_{\alpha\beta}(z)=H(z)+\overline{G(z)}
\end{align}
maps the unit disk $\mathbb{D}$ onto a domain convex in the horizontal direction.

Now, Theorem~\ref{Thm4.1P1} gives that 
$F^\theta_{\alpha\beta}$ given by \eqref{Eq6.5P1} is $\mathcal{SHU}$ for $\alpha,\beta$ satisfying the bound
\begin{equation}\label{Eq6.6P1}
|\alpha|\leq \frac{1}{2(2+|\beta|+2(1+|\beta|))},
\end{equation}
since in this case $\|w\|=1=\|w^*\|$.
While the second image of  Figure~\ref{Fig4P1} shows that there are non-stable harmonic univalent functions $F^\theta_{\alpha\beta}$
for $\alpha,\beta$ not satisfying \eqref{Eq6.6P1}, the first image of Figure~\ref{Fig4P1} demonstrates that $F^\theta_{\alpha\beta}$ is stable harmonic univalent function for $\alpha,\beta$ satisfying \eqref{Eq6.6P1}.
\begin{figure}[H]
\begin{minipage}[b]{0.45\textwidth}
\includegraphics[width=5.5cm]{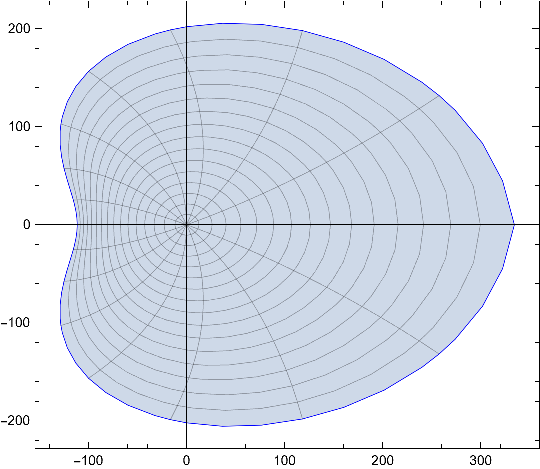}
\hspace*{0.7cm} For $\alpha=1/14$ and $\beta=1$
\end{minipage}
\begin{minipage}[b]{0.4\textwidth}
\includegraphics[width=5.5cm, height=5cm]{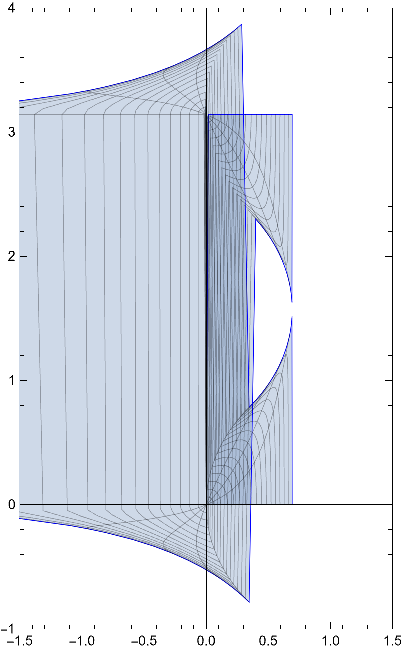}
\hspace*{1cm} For $\alpha=1/2$ and $\beta=1$
\end{minipage}
\caption{The image domains $F^\theta_{\alpha\beta}(\mathbb{D})$ for the above choices of $\alpha,\beta$.}\label{Fig4P1}
\end{figure}
\end{example}

\begin{example}
	We consider the analytic function $w(z)=\cos(\pi c)z/2$ and choose $\varphi(z)=z/(1-z)^2$ 
	in the definition of $C_{\alpha\beta}[\varphi]$ to obtain
	\begin{equation}\label{Eq6.7P1}
		C^\theta_{\alpha\beta}[\varphi](z)=e^{-i\theta}\int_{0}^{z} \frac{1}{(1-e^{i\theta}\zeta)^{\alpha (2+\beta)}}\, d\zeta=e^{-2i\theta}\bigg(\frac{1-(1-e^{i\theta}z)^{1-\alpha(2+\beta)}}{1-\alpha(2+\beta)}\bigg).
	\end{equation}
	Following the similar steps as explained in Example~\ref{Ex6.1P1}, 
one can easily obtain 
	$$
	H(z)=e^{-i\theta}\int_{0}^{z} \frac{2}{(1-e^{i\theta}\zeta)^{\alpha(2+\beta)}(2-\alpha(1+\beta)\cos(\pi c)z)}\, d\zeta 
	$$
and
$$G(z)=e^{-i\theta}\int_{0}^{z} \frac{\alpha(1+\beta)\cos(\pi c)z}{(1-e^{i\theta}\zeta)^{\alpha(2+\beta)}(2-\alpha(1+\beta)\cos(\pi c)z)}\, d\zeta
	$$
	under the usual normalization $H(0)=G(0)=0$. Therefore, the harmonic mapping 
	\begin{align}\label{Eq6.8P1}
		F^\theta_{\alpha\beta}(z)=H(z)+\overline{G(z)} 
	\end{align}
	maps the unit disk onto a domain convex in the horizontal direction.
	
	By using Theorem \ref{Thm5.1P1}, the mapping 
	$F^\theta_{\alpha\beta}$ given by \eqref{Eq6.8P1} is close-to-convex mapping for
	 all non-negative $\alpha,\beta$ satisfying 
	\begin{equation}\label{Eq6.9P1}
	\alpha(2+\beta)\le -2c. 
	\end{equation}
Figure~\ref{Fig5P1}'s first image illustrates the close-to-convexity of $F^\theta_{\alpha\beta}$
for $\alpha,\beta$ satisfying \eqref{Eq6.9P1} for $c=-4/10$, whereas Figure~\ref{Fig5P1}'s second image indicates the existence of non close-to-convex $F^\theta_{\alpha\beta}$ for $\alpha,\beta$ not satisfying \eqref{Eq6.9P1}.

	\begin{figure}[H]
		\begin{minipage}[b]{0.5\textwidth}
 			\includegraphics[height=4.1cm,width=6cm]{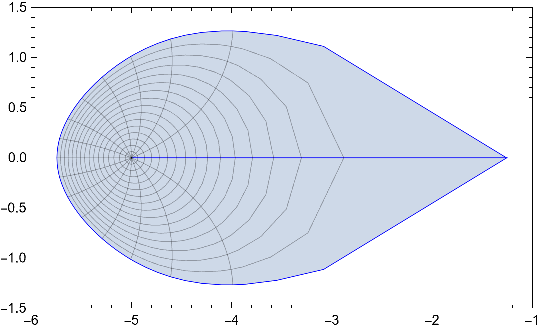}
			\hspace*{1.3cm} For $\alpha=4/15$ and $\beta=1$
		\end{minipage}
		\begin{minipage}[b]{0.45\textwidth}
			\includegraphics[width=5cm, height=4.1cm]{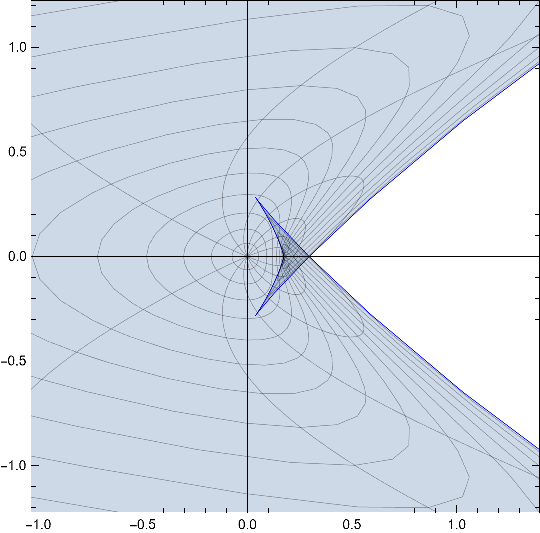}
			\hspace*{1.2cm} For $\alpha=-1$ and $\beta=1$
		\end{minipage}
		\caption{The image domains $F^\theta_{\alpha\beta}(\mathbb{D})$ for the above choices of $\alpha,\beta$.}\label{Fig5P1}
	\end{figure}
\end{example}

\begin{example}
Consider $w(z)=z/2$ and choosing $\varphi(z)=z/(1-z)$ 
in the definition of $C_{\alpha\beta}[\varphi]$, we obtain
\begin{equation}\label{Eq6.10P1}
C^\theta_{\alpha\beta}[\varphi](z)=e^{-i\theta}\int_{0}^{z} \frac{1}{(1-e^{i\theta}\zeta)^{\alpha (1+\beta)}}\, d\zeta=e^{-2i\theta}\bigg(\frac{1-(1-e^{i\theta}z)^{1-\alpha(1+\beta)}}{1-\alpha(1+\beta)}\bigg).
\end{equation}
Similar to the explanations used in Example~\ref{Ex6.1P1}, we find $F^\theta_{\alpha\beta}=H+\overline{G}$, where
$$
H(z)=e^{-i\theta}\int_{0}^{z} \frac{2}{(2-\alpha(1+\beta) \zeta)(1-e^{i\theta}\zeta)^{\alpha(1+\beta)}}\, d\zeta 
$$
and
$$
G(z)=e^{-i\theta}\int_{0}^{z} \frac{\alpha(1+\beta) \zeta}{(2-\alpha(1+\beta) \zeta)(1-e^{i\theta}\zeta)^{\alpha(1+\beta)}}\, d\zeta
$$
under the usual normalization $H(0)=G(0)=0$. Therefore, the harmonic mapping 
\begin{align}\label{Eq6.11P1}
F^\theta_{\alpha\beta}(z)=H(z)+\overline{G(z)} 
\end{align}
maps the unit disk onto a domain convex in the horizontal direction.

Inferred from Theorem \ref{Thm5.5P1} is that the mapping $F^\theta_{\alpha\beta}$ 
given by \eqref{Eq6.11P1} belongs to $\mathcal{SHCC}$ for all non-negative $\alpha,\beta$ satisfying
$\alpha \in [0, 1/(1+\beta)\sqrt{2}]$.
The stable harmonic close-to-convexity of $F^\theta_{\alpha\beta}$ is seen in the first image of Figure~\ref{Fig6P1} for $\alpha,\beta$ fulfilling the aforementioned limits, however the second image in Figure~\ref{Fig6P1} suggests the presence of non-stable harmonic close-to-convex $F^\theta_{\alpha\beta}$ for $\alpha,\beta$ not satisfying the aforementioned constraints.
\begin{figure}[H]
\begin{minipage}[b]{0.5\textwidth}
\includegraphics[width=5.7cm,height=4.1cm ]{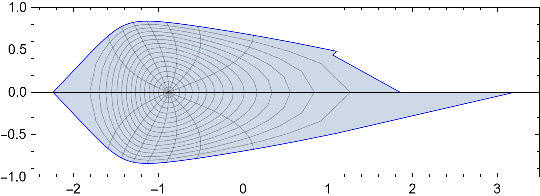}
\hspace*{0.8cm} For $\alpha=7/10$ and $\beta=1/100$
\end{minipage}
\begin{minipage}[b]{0.45\textwidth}
\includegraphics[width=5cm, height=4.1cm]{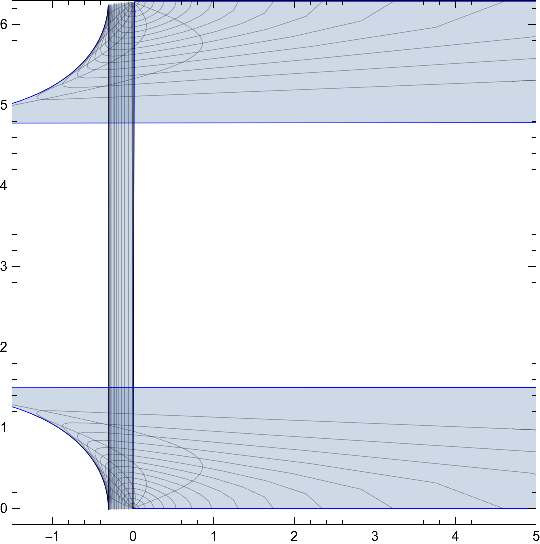}
\hspace*{0.8cm} For $\alpha=2$ and $\beta=-1/2$
\end{minipage}
\caption{The image domains $F^\theta_{\alpha\beta}(\mathbb{D})$ for the above choices of $\alpha,\beta$.}\label{Fig6P1}
\end{figure}
\end{example}

%%%%%%%%%%%%%%%%%%%%%%%%%%%%%%%%%%%%%%%%%%%%%%%%%%%%%%%%%%%%%%%%%%
%%%%%%%%%%%%%%%%%%%%%%%%%%%%%%%%%%%%%%%%%%%%%%%%%%%%%%%%%%%%%%%%%
%%%%%%%%%%%%%%%%%%%%% Section 7 %%%%%%%%%%%%%%%%%%%%%%%%%%%%%%%%%%
%%%%%%%%%%%%%%%%%%%%%%%%%%%%%%%%%%%%%%%%%%%%%%%%%%%%%%%%%%%%%%%%%%
%%%%%%%%%%%%%%%%%%%%%%%%%%%%%%%%%%%%%%%%%%%%%%%%%%%%%%%%%%%%%%%%%

\section{Concluding Remarks}
In the light of Theorem~\ref{thm3.4P1} and Remark~\ref{remark3.5P1}, the operator 
$C^\theta_{\alpha \beta}[\varphi]$ is CHD for all non-negative $\alpha,\beta$ satisfying $\alpha(\beta+2(1-\delta))\leq 3$ whenever $\varphi\in\mathcal{S}^*(\delta)$.
This has been used in Definition~\ref{Def3.5P1} and subsequently in the relevant results. It would be further interesting to concentrate on the problem for the remaining values of $\alpha$ and $\beta$.

We recall from \cite[Theorem~1]{HM74} that the Ces\`aro transform $C[\varphi]$
preserves the class $\mathcal{K}$. However, its corresponding harmonic mapping $F^0_{11}$ is not necessarily convex whenever $\varphi\in\mathcal{K}$. Indeed,    
by choosing $\varphi(z)=z/(1-z)$ and $w(z)=z/2$, 
we construct $F^0_{11}=H+\overline{G}$ with its dilatation $w_{11}(z)=z$.
Now we define an analytic function $\varPhi_{\lambda, 0}:=H+\lambda G$, $\lambda\in\mathbb{T}$, so that 
\begin{equation}\nonumber
\varPhi'_{\lambda,0}(z)=H'(z)\cdot [1+\lambda \,w_{11}(z)]=(C[\varphi])'(z) 
\cdot \frac{1+\lambda z}{1-z}.
\end{equation}
Thus, for all $z\in \mathbb{D}$, we compute 
$$
{\rm Re}\bigg[1+\frac{z \varPhi''_{\lambda,0}(z)}{\varPhi'_{\lambda,0}(z)}\bigg]={\rm Re} \bigg[1+\frac{3z}{1-z}+\frac{\lambda z}{1+\lambda z}\bigg].
$$
By choosing $z=-1/2$ and $\lambda=1$, we note that
$$
{\rm Re}\bigg[1+\frac{z\varPhi''_{\lambda, 0}(z)}{\varPhi'_{\lambda, 0}(z)}\bigg]=-1<0.
$$
Thus, by \cite[Theorem~3.1]{HM13-1}, 
$F^0_{11}=H+\overline{G}$ is not convex harmonic mapping in $\mathbb{D}$.

Following this, it is important to study the preserving property of $C^\theta_{\alpha\beta}[\varphi]$
when $\varphi\in\mathcal{K}$. This is seen in the proof of Theorem~\ref{Thm3.7P1}. 
Indeed, we notice that for all non-negative $\alpha,\beta$ 
with $\alpha(\beta+2(1-\delta))\leq 2$, the integral transform 
$C^\theta_{\alpha\beta}[\varphi]$ preserves the class $\mathcal{K}$.
However, it would be interesting
to find ranges of $\alpha$ and $\beta$ under which $F^\theta_{\alpha\beta}$
is convex whenever $\varphi\in\mathcal{K}$. This remains as an open problem. 

On the one side, the manuscript deals with the sufficient conditions for the univalence of
$F^\theta_{\alpha\beta}$ under certain constraints on $\alpha,\beta$, whereas on the other side, 
we observe from Section~\ref{Sec6P1} that there are non-univalent functions $F^\theta_{\alpha\beta}$ for some choices of $\alpha,\beta$ not satisfying such constraints. 
This observation suggests us to study the necessary conditions for the univalence 
of $F^\theta_{\alpha\beta}$ in terms of bounds of $\alpha$ and $\beta$, which remains open as well.

\section*{Acknowledgement(s)}

The authors of this manuscript are grateful to Professor S. Ponnusamy for giving up his time to speak with us on this subject.

\section*{Disclosure statement}

The authors declare that they have no conflict of interest.

\section*{Funding}

The work of the second author is supported by University Grants Commission -Ref. No.: 1163/(CSIR-UGC NET JUNE 2019). The authors also acknowledge the support of DST-FIST Project 
(File No.: SR/FST/MS I/2018/26) for providing research facilities in the department.

\end{document}